\UseRawInputEncoding
\documentclass[12pt, reqno]{amsart}
\usepackage{amssymb,latexsym,amsmath,amscd,amsfonts}
\usepackage{latexsym}
\usepackage[mathscr]{eucal}
\usepackage{bm}
\usepackage{mathptmx}

\voffset = -28pt \hoffset = -54pt \textwidth = 6.3in \textheight = 9.5in \numberwithin{equation}{section}

\newcommand{\beg}{\begin{equation}}
\newcommand{\eeg}{\end{equation}}
\newcommand{\ben}{\begin{eqnarray*}}
	\newcommand{\een}{\end{eqnarray*}}

\newtheorem{thm}{Theorem}[section]

\newtheorem{lem}[thm]{Lemma}

\newtheorem{prop}[thm]{Proposition}
\numberwithin{equation}{section} 
\theoremstyle{definition}
\newtheorem{defn}[thm]{Definition}
\newtheorem{rem}[thm]{Remark}

\newtheorem{eg}[thm]{Example}

\newcommand{\HS}{\mathcal H}
\newcommand{\C}{\mathbb{C}}

\newcommand{\D}{\mathbb{D}}
\newcommand{\T}{\mathbb{T}}
\newcommand{\ft}{\mathcal F_O}

\newcommand{\G}{{\Gamma}_n}

\newcommand{\gn}{\mathbb{G}_n}

\newcommand{\E}{\mathbb E}

\newcommand{\ov}{\overline}

\begin{document}

\title[Rational dilation on $\Gamma_n$ and distinguished varieties]
{Rational dilation for operators associated with spectral interpolation and distinguished varieties}

\author[Sourav Pal]{Sourav Pal}

\address[ Sourav Pal]{Department of Mathematics, Indian Institute of Technology Bombay,
Powai, Mumbai - 400076, India. } \email{sourav@math.iitb.ac.in,
souravmaths@gmail.com}



\thanks{The author is supported by the Seed Grant of IIT Bombay, the CPDA of the Govt. of India and the MATRICS Award of SERB, (Award No. MTR/2019/001010) of DST, India.}

\begin{abstract}

For $n\geq 2$, the symmetrized polydisc is the following family of polynomially convex domains:
\[
\mathbb G_n=\left\{ \left(\sum_{1\leq i\leq n} z_i,\sum_{1\leq
i<j\leq n}z_iz_j,\dots,
\prod_{i=1}^n z_i \right): \,|z_i|< 1, i=1,\dots,n \right \}.
\]
A distinguished variety in $\mathbb G_n$ is an affine algebraic variety that intersects $\mathbb G_n$ and exits $\mathbb G_n$ through its distinguished boundary $b\Gamma_n$, where $\Gamma_n=\overline{\mathbb G}_n$. A commuting tuple of Hilbert space operators $(S_1,\dots,S_{n-1},P)$, for which $\Gamma_n$ is a spectral set, is called a $\Gamma_n$-contraction. To every $\Gamma_n$-contraction, there is a unique operator tuple $(A_1, \dots , A_{n-1})$ acting on the defect space $\mathcal D_P$ of $P$ such that
$S_i-S_{n-i}^*P=D_PA_iD_P$ for $i = 1, \dots , n-1$. This unique tuple is called the \textit{fundamental operator tuple} or simply the $\mathcal F_O$-\textit{tuple} of the $\Gamma_n$-contraction $(S_1, \dots , S_{n-1},P)$. Let $\mathcal K$ be the minimal unitary dilation space of the contraction $P$. We obtain the following main results in this paper.

\begin{enumerate}

\item We characterize the entire class of $\Gamma_n$-contractions $(S_1, \dots , S_{n-1},P)$ which possess normal $b\Gamma_n-$dilation (i.e. $\Gamma_n$-unitary dilation) on $\mathcal K$ with last component being the minimal unitary dilation of $P$. We show explicit construction of such dilation.

\item We show that if $\mathcal D_P$ is finite dimensional, then $(S_1, \dots , S_{n-1},P)$ possesses a normal $b\Gamma_n-$dilation on $\mathcal K$ if and only if the $\mathcal F_O$-tuples of $(S_1, \dots , S_{n-1},P)$ and $(S_1^*, \dots , S_{n-1}^*,P^*)$ define distinguished varieties in $\mathbb G_n$.

\item We prove that if $P$ is a pure contraction and if $\mathcal D_{P^*}$ is finite dimensional, then a $\Gamma_n$-contraction $\Sigma=(S_1, \dots , S_{n-1},P)$ has a normal $b\Gamma_n-$dilation if and only if it has a normal $\partial \overline{\Lambda}_{\Sigma}-$dilation, where $\Lambda_{\Sigma}$ is a distinguished variety in $\mathbb G_n$ determined by the $\mathcal F_O$-tuple of $(S_1^*, \dots , S_{n-1}^*,P^*)$ and further establish that this is equivalent to the existence of the distinguished variety $\Lambda_{\Sigma}$.

\item We find a new characterization for the distinguished varieties in the symmetrized polydisc.

\item We produce a concrete functional model for a $\Gamma_n$-contraction which dilates to a $\Gamma_n$-unitary on $\mathcal K$.

\item We find new characterizations for the points in $\mathbb G_n$ and $\Gamma_n$.

\item We obtain an operator model for a pure $\Gamma_n$-isometry and find new characterizations for the $\Gamma_n$-unitaries and $\Gamma_n$-isometries.

\end{enumerate}

\end{abstract}

\maketitle

\tableofcontents

\section{Introduction}

\vspace{0.4cm}

\noindent Throughout the paper all operators are bounded linear operators acting on complex Hilbert spaces. A contraction is an operator with norm not greater than one. Denote by $\D, \T$ the unit disc and the unit circle respectively in the complex plane $\C$. We define spectral set, complete spectral set, distinguished boundary and rational dilation in Section \ref{background}.

The aim of this article is to explore and establish a connection between rational dilation and the geometry of the distinguished varieties in the symmetrized polydisc. The \textit{symmetrized} $n$-\textit{disc} or simply the \textit{symmetrized polydisc} $\mathbb G_n$, which is defined by
\[
\mathbb G_n =\left\{ \left(\sum_{1\leq i\leq n} z_i,\sum_{1\leq
i<j\leq n}z_iz_j,\dots, \prod_{i=1}^n z_i \right): \,|z_i|< 1,
i=1,\dots,n \right \},
\]
is family of polynomially convex domains that are naturally associated with the famous Spectral interpolation problem. If $\mathcal M_n(\mathbb C)$ is the space of $n\times n$ complex matrices and if $\mathcal B_1^n$ is its spectral unit ball, then $A\in \mathcal B_1^n$ (that is, the spectral radius $r(A)<1$) if and only if $\pi_n(\sigma(A)) \in \mathbb G_n$, where $\sigma(A)$ is the spectrum of $A$ and $\pi_n:\mathbb C^n \rightarrow \mathbb C^n$ is the symmetrization map whose components are the symmetric polynomials in $n$ variables, i.e.
\[
\pi_n(z)=\left(\sum_{1\leq i\leq n} z_i,\sum_{1\leq
i<j\leq n}z_iz_j,\dots, \prod_{i=1}^n z_i \right).
\]
Note that $\mathbb G_1$ is the unit disc $\D$ and for $n\geq 2$, $\mathbb G_n$ is the symmetrization of the points in the polydisc $\D^n$, i.e. $\gn=\pi_n(\D^n)$. Here we shall consider $\gn$ for $n\geq 2$ only. Naturally a bounded domain like $\mathbb G_n$ that has complex-dimension $n$, is much easier to deal with than an unbounded $n^2$-dimensional object like $\mathcal B_1^n$. Nevertheless, the symmetrized polydisc has evloved as an object of independent inetrests in past two decades because of its rich function theory \cite{costara1, KZ, MSZ, NiPa, Pal-Roy2}, uncanonical complex geometry \cite{costara, edi-zwo, NiPfZw, N-P-T-Zw1, Pal-Roy3, Su-T-Tu} and more appealing operator theory \cite{ay-jfa, tirtha-sourav, tirtha-sourav1, BSR, sourav, sarkar, Bisai-Pal1, Bisai-Pal2} along with their fascinating interrelations, e.g. \cite{pal-shalit, spal2}. Also, an interested reader can see the references therein for a further reading.\\

The goal of rational dilation roughly speaking is to realize a
tuple of commuting operators as a compression of a tuple of commuting normal operators. If a compact set $K \subset \mathbb C^d$ is a spectral set for a commuting tuple $(T_1,\cdots,T_d)$ acting on a Hilbert space $\HS$, the rational dilation seeks the existence of a commuting tuple of normal operators $(N_1, \dots , N_d)$ acting on a Hilbert space $\widetilde{\HS} \supseteq \HS$ and having the boundary of $K$ as a spectral set such that
\[
f(T_1, \dots , T_d)=P_{\HS}\, f(N_1, \dots , N_d)|_{\HS}
\]
for all rational functions $f=p/q$, where $P_{\HS}:\widetilde{\HS} \rightarrow \HS$ is the orthogonal projection and $p, q \in \C[z_1, \dots , z_d]$ with $q$ having no zeros in $K$. The first nontrivial step in this direction was Sz.-Nagy's dilation theorem which removed much of the mystery of one variable operator theory by expressing a contraction as a compression of a unitary operator.
\begin{thm}[Sz.-Nagy, 1953]
If $T$ is a contraction acting on a Hilbert space $\mathcal H$,
then there exists a Hilbert space $\widetilde{\HS} \supseteq \mathcal
H$ and a unitary $U$ on $\widetilde{\HS}$ such that
\[
r(T)=P_{\mathcal H}r(U)|_{\mathcal H}
\]
for all $r\in Rat(\ov{\D})$, where $Rat(\ov{\D})$ is the algebra of rational functions $r(z)$ with poles off $\ov{\D}$.
\end{thm}
Moreover, such a unitary dilation is called \textit{minimal} if
$
\widetilde{\HS}=\ov{span} \{ r(T): r\in Rat(\ov{\D}) \}.
$
The theory of Sz.-Nagy and Foias \cite{nagy} tells us that any two minimal unitary dilations of a contraction are unitarily equivalent which is why often it is referred to as the minimal unitary dilation. Note that an operator is a contraction if and only if $\ov{\D}$ is a spectral set for it and a unitary is a normal operator having the boundary $\T$ as a spectral set. Since every operator is nothing but a scalar
time a contraction, Sz.-Nagy's result provides a subtle way of modelling an operator in terms of a normal operator or more precisely a scalar time a unitary.

 Later, Ando profoundly generalized Sz.-Nagy's result in \cite{ando} by exhibiting a commuting isometric dilation to a pair of commuting contractions. This makes $\ov{\D}^2$ a spectral set for any pair of commuting contractions $(T_1,T_2)$ and consequently we have von-Neumann's inequality which is known as Ando's inequality:
\[
\|f(T_1,T_2)\| \leq \sup_{(z_1,z_2)\in \ov{\D}^2}\; |f(z_1,z_2)|=\|f\|_{\infty, \ov{\D}^2},
\]
where $f=p/q$ is any rational function such that $q$ has no zeros in $\ov{\D}^2$. Also, the literature \cite{parrott} tells us that a commuting tuple of contractions $(T_1, \dots ,T_n)$ may not dilate to a tuple of commuting unitaries $(U_1, \dots , U_n)$ if $n>2$. In \cite{AM05}, Agler and M\raise.45ex\hbox{c}Carthy showed that if $(Q_1,Q_2)$ is a pair of commuting contractive matrices without having any unimodular eigenvalues, then there is an algebraic variety $V$ intersecting $\D^2$ and satisfying $V\cap \ov{\D}^2=V\cap (\D^2 \cup \T^2)$ such that Ando's inequality holds for $(Q_1,Q_2)$ on $V\cap \ov{\D}^2$. The set $V\cap \D^2$ is called a \textit{distinguished variety} in $\D^2$. This path-breaking result reduces a spectral set for $(Q_1,Q_2)$ by a dimension in the sense that $\ov{\D}^2$ is naturally a spectral set for $(Q_1,Q_2)$ and has complex dimension $2$ whereas ${V\cap \ov{\D}^2}$ has complex dimension $1$ and it is also a spectral set for $(Q_1,Q_2)$ by Ando's inequality. In past one decade the distinguished varieties in domains like the polydisc, the symmetrized polydisc and the tetrablock have been extensively studied, e.g. \cite{pal-shalit, spal01, Bh-Ku-Sau, spal1, spal2}.
\begin{defn}
Let $G\subseteq \C^n$ be a domain. A nonempty set $W\subseteq G$ is said to be a \textit{distinguished variety} in $G$ if there is a complex affine variety $W'$ in $\C^n$ such that $W=W'\cap G$ and $W'\cap \partial \ov{G} \subseteq b\ov{G}$, where $\partial \ov{G},\; b\ov{G}$ are the topological and distinguished boundary of $G$ respectively. We denote by $\partial \ov{W}$ the set $W'\cap \partial \ov{G}= W' \cap b \ov{G}$.
\end{defn}
It was shown in \cite{spal2} by the author that any distinguished variety in $\gn$ could be determined by a set of polynomials of the form $\{ \det \, (F_i^*+z_n F_{n-i}-z_iI)\,:\,1 \leq i \leq n-1 \}$, where $F_1, \dots , F_{n-1}$ are complex square matrices of same order satisfying certain conditions. This naturally leads us to the following definition.
\begin{defn}
A set of $n-1$ complex square matrices of same order $\{ F_1, \dots , F_{n-1} \}$ is said to define a distinguished variety $W=W'\cap \gn$ in the symmetrized polydisc, where $W'$ is an affine variety in $\C^n$ if
\[
W = \{ (s_1,\dots,s_{n-1},p)\in \mathbb G_n \,: \, (s_1,\dots,s_{n-1}) \in \sigma_T(F_1^*+pF_{n-1}\,,\,
F_2^*+pF_{n-2}\,,\,\dots\,, F_{n-1}^*+pF_1) \},
\]
so that $\{ \det \, (F_i^*+z_n F_{n-i}-z_iI)\,:\,1 \leq i \leq n-1 \}$ becomes a set of generators for $W'$.
\end{defn}

Let us consider an operator tuple for which the closed symmetrized polydisc is a spectral set.
\begin{defn}
A commuting tuple of Hilbert space operators $(S_1, \dots , S_{n-1},P)$ having the closed symmetrized polydisc
\[
\G =\left\{ \left(\sum_{1\leq i\leq n} z_i,\sum_{1\leq
i<j\leq n}z_iz_j,\dots,
\prod_{i=1}^n z_i \right): \,|z_i|\leq 1, i=1,\dots,n \right \}
\]
as a spectral set is called a $\G$-\textit{contraction}.
\end{defn}
It follows from the definition that if $(S_1, \dots , S_{n-1},P)$ is a $\G$-contraction then so is its adjoint $(S_1^*, \dots , S_{n-1}^*,P^*)$ and that $P$ is a contraction. It was shown in \cite{sourav6} by the author that for any $\G$-contraction $(S_1, \dots , S_{n-1},P)$, there are unique operators $A_1, \dots , A_{n-1}$ acting on the defect space $\mathcal D_P=\ov{Ran}\,(I-P^*P)^{\frac{1}{2}}$ of $P$ such that
\[
S_i-S_{n-i}^*P=(I-P^*P)^{\frac{1}{2}}A_i(I-P^*P)^{\frac{1}{2}}\,, \quad 1\leq i \leq n-1.
\]
This unique operator tuple $(A_1, \dots , A_{n-1})$, for its pivotal role in the theory of $\G$-contractions, is called the \textit{fundamental operator tuple} or simply the $\ft$-\textit{tuple} of $(S_1, \dots , S_{n-1},P)$.\\

In Theorem \ref{main-dilation-theorem}, we find a necessary and sufficient condition for $(S_1, \dots , S_{n-1},P)$ to possess a normal $b\Gamma_n-$dilation (i.e. a $\G$-unitary dilation) $(R_1, \dots , R_{n-1},U)$ such that $U$ is the minimal unitary dilation of the contraction $P$. Also, we explicitly construct such a $\G$-unitary dilation. The $\ft$-tuples of $(S_1, \dots , S_{n-1},P)$ and $(S_1^*, \dots , S_{n-1}^*,P^*)$ play central role in that construction. We show by examples in Subsection \ref{further-dilation} that a $\G$-contraction may dilate to a $\G$-unitary without satisfying the conditions of Theorem \ref{main-dilation-theorem}. The dilation issue is subtle when the defect space of $P$ is finite dimensional. Indeed, in Theorem \ref{thm:dilation-variety2} we show that if $\mathcal D_P$ is finite dimensional then $(S_1, \dots , S_{n-1},P)$ dilates to a $\G$-unitary on $\mathcal K$ if and only if the $\ft$-tuples of $(S_1, \dots , S_{n-1},P)$ and $(S_1^*, \dots , S_{n-1}^*,P^*)$ define distinguished varieties in $\gn$. The interplay between dilation and distinguished varieties becomes more interesting when $P$ is a pure contraction, i.e. ${P^*}^n \rightarrow 0$ strongly and $\mathcal D_{P^*}$ is finite dimensional. It is well-known that $L^2(\mathcal D_{P^*})$ is the minimal unitary dilation space of $P$ then. In Theorem \ref{thm:dilation-variety1}, we show that if $P$ is pure with $\dim \mathcal D_{P^*}< \infty$ and if $(B_1, \dots , B_{n-1})$ is the $\ft$-tuple of $(S_1^*, \dots , S_{n-1}^*,P^*)$, then $(S_1, \dots , S_{n-1},P)$ admits a normal $b\G$-dilation on $L^2(\mathcal D_{P^*})$ if and only if it has a normal $\partial \ov{\Lambda}-$dilation, where $\Lambda$ is a distinguished variety in $\gn$ generated by the set of polynomials $\{\det \, (B_i^*+z_nB_{n-i}-z_iI):\, 1 \leq i \leq n-1 \}$. This gives rise to a new characterization for the distinguished varieties in $\gn$ which we present in Theorem \ref{thm:DVchar-4}. Note that the author's previous work \cite{spal2} gives characterization of a distinguished variety in $\gn$ in terms of Taylor joint spectrum of $n-1$ commuting matrix pencils. The new characterization that we obtain here is simpler than the earlier one.

In Section \ref{sec:isometric-dilation-1}, we find a concrete functional model for a $\G$-contraction that dilates to a $\G$-unitary on $\mathcal K$. To do so we produce an $\G$-isometric dilation acting on the minimal isometric dilation space of $P$. An abstract operator model for a pure $\G$-isometry was described in \cite{BSR} in terms of Toeplitz operators on a vectorial Hardy space. In Theorem \ref{model1}, we make a refinement of that model. We show that if $(T_1, \dots , T_{n-1},V)$ is a pure $\G$-isometry on $\HS$ then $\HS$ can be identified with $H^2(\mathcal D_{V^*})$ and $(T_1, \dots , T_{n-1},V)$ can be identified with the Toeplitz operator tuple $(T_{\phi_1}, \dots , T_{\phi_{n-1}},T_z)$, where $\phi_i(z)=F_i^*+F_{n-i}z$ with $(F_1, \dots , F_{n-1})$ being the $\ft$-tuple of $(T_1^*, \dots , T_{n-1}^*,V^*)$. In Section \ref{scalarcase}, we find a few new characterizations for the points in $\gn$ and $\G$. The other achievement of this paper is a set of new characterizations for the $\G$-unitaries and the $\G$-isometries which we present in Section \ref{unitaries-isometries}. In Section \ref{background}, we recall some basic facts and accumulate a few useful results from the literature.

We learned that parts of the characterizations for the points in $\gn, \G$ and the same for the $\G$-unitaries and $\G$-isometries are independently found by A. Pal in \cite{a-pal}.

\vspace{0.4cm}

\section{Basic definitions and preparatory results}\label{background}

\vspace{0.4cm}

\subsection{The Taylor joint spectrum}

Let $\Lambda$ be the exterior algebra on $n$ generators
$e_1,...e_n$, with identity $e_0\equiv 1$. $\Lambda$ is the
algebra of forms in $e_1,...e_n$ with complex coefficients,
subject to the collapsing property $e_ie_j+e_je_i=0$ ($1\leq i,j
\leq n$). Let $E_i: \Lambda \rightarrow \Lambda$ denote the
creation operator, given by $E_i \xi = e_i \xi $ ($\xi \in
\Lambda, 1 \leq i \leq n$).
 If we declare $ \{ e_{i_1}... e_{i_k} : 1 \leq i_1 < ... < i_k \leq n \}$ to be an
 orthonormal basis, the exterior algebra $\Lambda$ becomes a Hilbert space,
 admitting an orthogonal decomposition $\Lambda = \oplus_{k=1} ^n \Lambda^k$
 where $\dim \Lambda ^k = {n \choose k}$. Thus, each $\xi \in \Lambda$ admits
 a unique orthogonal decomposition
 $ \xi = e_i \xi' + \xi''$, where $\xi'$ and $\xi ''$ have no $e_i$
contribution. It then follows that that  $ E_i ^{*} \xi = \xi' $,
and we have that each $E_i$ is a partial isometry, satisfying
$E_i^*E_j+E_jE_i^*=\delta_{i,j}$. Let $\mathcal X$ be a normed
space, let $\underline{T}=(T_1,...,T_n)$ be a commuting $n$-tuple
of bounded operators on $\mathcal X$ and set $\Lambda(\mathcal
X)=\mathcal X\otimes_{\mathbb{C}} \Lambda$. We define
$D_{\underline T}: \Lambda (\mathcal X) \rightarrow \Lambda
(\mathcal X)$ by

\[
D_{\underline T} = \sum_{i=1}^n T_i \otimes E_i .
\]

Then it is easy to see $D_{\underline T}^2=0$, so $Ran
D_{\underline T} \subset Ker D_{\underline T}$. The commuting
$n$-tuple is said to be \textit{non-singular} on $\mathcal X$ if
$Ran D_{\underline T}=Ker D_{\underline T}$.
\begin{defn}
The Taylor joint spectrum of ${\underline T}$ on $\mathcal X$ is
the set
\[
\sigma_T({\underline T},\mathcal X) = \{
\lambda=(\lambda_1,...,\lambda_n)\in \mathbb{C}^n : {\underline
T}-\lambda \text{ is singular} \}.
\]
\end{defn}
\begin{rem}
The decomposition $\Lambda=\oplus_{k=1}^n \Lambda^k$ gives rise to
a cochain complex $K({\underline T},\mathcal X)$, known as the
Koszul complex associated to ${\underline T}$ on $\mathcal X$, as
follows:
\[
K({\underline T},\mathcal X):0 \rightarrow \Lambda^0(\mathcal
X)\xrightarrow{D_{\underline T}^0}... \xrightarrow{D_{\underline
T}^{n-1}} \Lambda^n(\mathcal X) \rightarrow 0 ,
\]
where $D_{\underline T}^{k}$ denotes the restriction of
$D_{\underline T}$ to the subspace $\Lambda^k(\mathcal X)$. Thus,
\[
\sigma_T({\underline T},\mathcal X) = \{ \lambda\in \mathbb{C}^n :
K({\underline T}-\lambda ,\mathcal X)\text{ is not exact} \}.
\]
\end{rem}
For a further reading on Taylor joint spectrum an interested
reader is referred to Taylor's works, \cite{Taylor, Taylor1}.

\subsection{Spectral and complete spectral set}

We shall follow Arveson's terminologies about the spectral and
complete spectral sets. Let $X$ be a compact subset of $\mathbb
C^n$ and let $\mathcal R(X)$ denote the algebra of all rational
functions on $X$, that is, all quotients $p/q$ of polynomials $p,q$ from $\C[z_1, \dots , z_n]$ for which $q$ has no zeros in $X$. The norm of an element
$f$ in $\mathcal R(X)$ is defined as
$$\|f\|_{\infty, X}=\sup \{|f(\xi)|\;:\; \xi \in X  \}. $$
Also for each $k\geq 1$, let $\mathcal R_k(X)$ denote the algebra
of all $k \times k$ matrices over $\mathcal R(X)$. Obviously each
element in $\mathcal R_k(X)$ is a $k\times k$ matrix of rational
functions $F=(f_{i,j})$ and we can define a norm on $\mathcal
R_k(X)$ in the canonical way
$$ \|F\|=\sup \{ \|F(\xi)\|\;:\; \xi\in X \}, $$ thereby making
$\mathcal R_k(X)$ into a non-commutative normed algebra. Let
$\underline{T}=(T_1,\cdots,T_n)$ be an $n$-tuple of commuting
operators on a Hilbert space $\mathcal H$. The set $X$ is said to
be a \textit{spectral set} for $\underline T$ if the Taylor joint
spectrum $\sigma_T (\underline T)$ of $\underline T$ is a subset
of $X$ and
\begin{equation}\label{defn1}
\|f(\underline T)\|\leq \|f\|_{\infty, X}\,, \textup{ for every }
f\in \mathcal R(X).
\end{equation}
Here $f(\underline T)$ can be interpreted as $p(\underline
T)q(\underline T)^{-1}$ when $f=p/q$. Moreover, $X$ is said to be
a \textit{complete spectral set} if $\|F(\underline T)\|\leq
\|F\|$ for every $F$ in $\mathcal R_k(X)$, $k=1,2,\cdots$.\\

The following two results are well-known and one can find proofs of them in \cite{spal2}.

\begin{lem}\label{imp-lem:01}
Let $(T_1,\dots,T_d)$ be a commuting tuple of normal operators on a Hilbert space $\mathcal H$. Then a compact set $K \subseteq \mathbb C^d$ is a spectral set for $(T_1,\dots,T_d)$ if and only if the Taylor joint spectrum $\sigma_T((T_1,\dots,T_d)) \subseteq K$.
\end{lem}

\begin{lem}\label{imp-lem:02}
Let $(T_1,\dots,T_d)$ be a commuting tuple of operators on a Hilbert space $\mathcal H$. Then a polynomially convex set $K \subseteq \mathbb C^d$ is a spectral set for $(T_1,\dots,T_d)$ if and only if
\[
\| f(T_1,\dots,T_d) \|\leq \| f \|_{\infty, K}
\]
for all polynomials $f\in \mathbb C[z_1,\dots,z_d]$.
\end{lem}

\subsection{The distinguished boundary and rational dilation}

Let $\mathcal A(X)$ be an algebra of continuous complex-valued
functions on $X$ which separates the points of $X$. A
\textit{boundary} for $\mathcal A(X)$ is a closed subset $\Delta$ of $X$ such that every function in $\mathcal A(X)$ attains its
maximum modulus on $\Delta$. It follows from the theory of
uniform algebras that the intersection of all the boundaries of
$X$ is also a boundary for $\mathcal A(X)$ (see Theorem 9.1 of
\cite{wermer}). This smallest boundary is called the
$\check{\textup{S}}$\textit{ilov boundary} for $\mathcal A(X)$.
When $\mathcal A(X)$ is the algebra of rational functions which
are continuous on $X$, the $\check{\textup{S}}$\textit{ilov
boundary} for $\mathcal A(X)$ is called the \textit{distinguished
boundary} of $X$ and is denoted by $bX$.

A commuting $n$-tuple of operators $\underline T$ on a Hilbert
space $\mathcal H$, having $X$ as a spectral set, is said to have
a \textit{rational dilation} or \textit{normal}
$bX$-\textit{dilation} if there exists a Hilbert space $\mathcal
K$, an isometry $V:\mathcal H \rightarrow \mathcal K$ and an
$n$-tuple of commuting normal operators $\underline
N=(N_1,\cdots,N_n)$ on $\mathcal K$ with $\sigma_T(\underline
N)\subseteq bX$ such that
\begin{equation}\label{rational-dilation}
f(\underline T)=V^*f(\underline N)V, \textup{ for every } f\in
\mathcal R(X),
\end{equation}
or, in other words $f(\underline T)=P_{\mathcal H}f(\underline N)|_{\mathcal H}$ for every $f\in \mathcal R(X)$ when $\mathcal H$ is considered as a closed linear subspace of $\mathcal K$. Moreover, the dilation is called {\em minimal} if
\[
\mathcal K=\overline{\textup{span}}\{ f(\underline N) h\,:\;
h\in\mathcal H \textup{ and } f\in \mathcal R(K) \}.
\]
It obvious that if $X$ is a complete spectral set for $\underline
T$ then $X$ is a spectral set for $\underline T$. A celebrated
theorem of Arveson states that $\underline T$ has a normal
$bX$-dilation if and only if $X$ is a complete spectral set of
$\underline T$ (Theorem 1.2.2 and its corollary, \cite{sub2}).
Therefore, the success or failure of rational dilation is
equivalent to asking whether $X$ is a spectral set for $\underline
T$ implies that $X$ is a complete spectral set for $\underline
T$.

\subsection{A brief literature and preparatory results} Recall that a $\G$-contraction is a commuting tuple $(S_1, \dots , S_{n-1},P)$ that has $\G$ as a spectral set. Unitaries, isometries and co-isometries are important special classes of contractions. There are natural analogues of these classes for $\Gamma_n$-contractions.
\begin{defn}
Let $S_1,\dots,S_{n-1},P$ be commuting operators on a Hilbert space
$\mathcal H$. We say that $(S_1,\dots,S_{n-1},P)$ is
\begin{itemize}
\item [(i)] a $\Gamma_n$-\textit{unitary} if $S_1,\dots,S_{n-1},P$ are
normal operators and $\sigma_T(S_1,\dots,S_{n-1},P) \subseteq b\Gamma_n$ ; \item [(ii)] a $\Gamma_n$-\textit{isometry} if there exists a Hilbert space
$\mathcal K \supseteq \mathcal H$ and a $\Gamma_n$-unitary
$(\tilde{S_1},\dots,\tilde{S_{n-1}},\tilde{P})$ on $\mathcal K$ such that $\mathcal H$ is a common invariant subspace for
$\tilde{S_1},\dots,\tilde{S_{n-1}},\tilde{P}$ and that
$S_i=\tilde{S_i}|_{\mathcal H}$ for $i=1,\dots,n-1$ and
$\tilde{P}|_{\mathcal H}=P$; \item [(iii)] a
$\Gamma_n$-\textit{co-isometry} if $(S_1^*,\dots,S_{n-1}^*,P^*)$ is a $\Gamma_n$-isometry.
\end{itemize}
\end{defn}

\begin{defn}
A $\Gamma_n$-isometry $(S_1,\dots,S_{n-1},P)$ is said to be \textit{pure} if $P$ is a pure isometry, i.e, if ${P^*}^n \rightarrow 0$ strongly as $n \rightarrow \infty$.
\end{defn}
In \cite{sourav3}, we introduced the following $n-1$ operator pencils
$\Phi_1,\dots,\Phi_{n-1}$ to analyze the structure of a
$\Gamma_n$-contraction:
\begin{equation} \label{eq:1a}
\Phi_{i}(S_1,\dots, S_{n-1},P)
={\tilde n_i}^2(I-P^*P)+(S_i^*S_i-S_{n-i}^*S_{n-i})-\tilde n_i(S_i-S_{n-i}^*P) -\tilde n_i(S_i^*-P^*S_{n-i}).
\end{equation}
where $\tilde n_i= \binom{n}{i}$ and $S_i, P$ are commuting operators with $\|P\|\leq 1$. Note that these operator pencils are positive when $(S_1, \dots , S_{n-1},P)$ is a $\G$-contraction.
\begin{prop}[\cite{sourav3}, Proposition 2.6]\label{lem:3}
Let $(S_1,\dots,S_{n-1},P)$ be a $\Gamma_n$-contraction. Then for
$i=1,\dots,n-1,\; \Phi_i(\alpha
S_1,\dots,\alpha^{n-1}S_{n-1},\alpha^n P)\geq 0$ for all $\alpha
\in\overline{\mathbb D}$.
\end{prop}
By an application of positivity of $\Phi_1, \dots , \Phi_{n-1}$, we established in \cite{sourav6} the existence and uniqueness of the $\ft$-tuple of a $\G$-contraction.
\begin{thm}[\cite{sourav6}, Theorem 3.3]\label{existence-uniqueness} Let
$(S_1,\dots,S_{n-1},P)$ be a $\Gamma_n$-contraction on a Hilbert
space $\mathcal H$. Then there are unique operators
$A_1,\dots,A_{n-1}\in\mathcal B(\mathcal D_P)$ such that
\[
S_i-S_{n-i}^*P=D_PA_iD_P \text{ for } i=1,\dots,n-1.
\]
Moreover, for each $i$ and for all $z\in \mathbb T$, $\omega
(A_i+A_{n-i}z)\leq \tilde n_i$. $[\tilde n_i= \binom{n}{i}]$.
\end{thm}
The following result from \cite{BSR} will be frequently used in the entire paper.
\begin{lem}[\cite{BSR}, Lemma 2.7] \label{lem:BSR2}
For any $\G$-contraction $(S_1,\dots,S_{n-1},P)$, the operator tuple
$
\left(\dfrac{n-1}{n}S_1,\dfrac{n-2}{n}S_2, \dots, \dfrac{1}{n}S_{n-1} \right)
$
is a $\Gamma_{n-1}$-contraction.
\end{lem}

\vspace{0.4cm}

\section{The symmetrized polydisc, its closure and the distinguished boundary}\label{scalarcase}

\vspace{0.4cm}

\noindent In this section, we present a few new characterizations for the points in the open and closed symmetrized polydiscs. We begin with $\gn$.

\begin{thm}\label{thm:SC2}
Let $(s_1,\dots,s_{n-1},p)\in\mathbb C^n$ and let
\[
Q=\left( \dfrac{s_1-\overline{ s_{n-1}}p}{1-|p|^2}, \dfrac{s_2-\overline{ s_{n-2}}p}{1-|p|^2},\dots,
\dfrac{s_{n-1}-\overline{ s_{1}}p}{1-|p|^2} \right)\in\mathbb C^{n-1}\,,
\]
when $|p|\neq 1$. Then the following are equivalent:
\begin{enumerate}
\item $(s_1,\dots,s_{n-1},p)\in\mathbb G_n$\,;  \item $(\omega
s_1,\omega^2 s_2,\dots, \omega^{n-1}s_{n-1},\omega^n p)\in\mathbb
G_n$ for all $\omega \in\mathbb T$  \,; \item $\Phi_{i}(\alpha
s_1,\dots, \alpha^{n-1} s_{n-1},\alpha^n p)> 0$ for all
$\alpha\in\overline{\mathbb D}$, $i=1,2, \dots, n-1$ and
$Q\in\mathbb G_{n-1}$ \,; \item $\left| {\tilde n_i \alpha^n
p-\alpha^{n-i} s_{n-i}} \right|< |{\tilde n_i-\alpha^i s_i}|$, for
all $\alpha\in\overline{\mathbb D}$ and for $i=1,\dots,n-1$. and
$Q\in\mathbb G_{n-1}$ \,; \item
$|s_i-\bar{s_{n-i}}p|+|s_{n-i}-\bar{s_i}p|< n(1-|p|^2)$,
$Q\in\mathbb G_{n-1}$ \,; \item $|p|< 1$ and there exists
$(c_1,\dots,c_{n-1})\in\mathbb G_{n-1}$ such that
\[
s_i=c_i+\overline{{c}_{n-i}}p\;, \quad i=1,\dots,n-1.
\]
\end{enumerate}

\end{thm}

\begin{proof}
The equivalence of $(1)$ and $(6)$ was established in
\cite{costara1} (see Theorem 3.6 in \cite{costara1}). Also
$(1)\Leftrightarrow (2)$ follows from author's previous result
Lemma 2.3 in \cite{sourav3}.
We shall prove here $(1)\Rightarrow (3)\Rightarrow (5) \Rightarrow (6)$ and $(3)\Leftrightarrow (4)$.\\

\noindent $(1)\Rightarrow (3).$ Let $(s_1,\dots,s_{n-1},p)\in
\mathbb G_n$ and let $\alpha\in \overline{\mathbb D}$. Then
\[
(\alpha s_1,\dots,{\alpha}^{n-1}s_{n-1},{\alpha}^np)\in\mathbb
G_n.
\]
The fact that
\[
\Phi_{i}(\alpha s_1,\dots,\alpha^{n-i} s_{n-i}, \alpha^n p)>0
\]
was proved in Proposition 2.5 in \cite{sourav3}. Also one can
clearly see that it follows from operator version of the result,
that is, Proposition \ref{lem:3} as each point in $\mathbb G_n$ is
a $\Gamma_n$-contraction. We have that
\[
s_i=\left( \frac{s_i-\overline{s_{n-i}p}}{1-|p|^2} \right) +
\left( \overline{\frac{s_{n-i}-\bar s_i p}{1-|p|^2}} \right) p =
c_i+\overline{c_{n-i}}p.
\]
By Theorem 3.6 in \cite{costara1}, $(c_1,\dots,c_{n-1})\in\mathbb
G_{n-1}$ is unique such that
\[
s_i=c_i+\overline{c_{n-i}}p \quad \text{ for each } i\,.
\]
So, we have that $Q=(c_1,\dots,c_{n-1})\in\mathbb G_{n-1}$.\\

\noindent $(3)\Rightarrow (5).$ First we assume that $s_i\neq
\overline{s_{n-i}}p$ for each $i=1,\dots,n-1$. Then for $\omega_1 ,\omega_2
\in\mathbb T$, we have for each $i=1,\dots,n-1$
\[
\Phi_{i}(\omega_1 s_1,\dots,{\omega_1}^{n-1}
s_{n-1},{\omega_1}^np)+\Phi_{n-i}(\omega_2 s_1,\dots, {\omega_2}^{n-1}
s_{n-1},{\omega_2}^np) > 0
\]
that is
\[
\tilde n_i^2(1-|p|^2)> \text{Re }\tilde n_i
[\omega_1(s_i-\overline{s_{n-i}}p)+{\omega_2}^2(s_{n-i}-\overline{s_i}p)].
\]
Since $s_i\neq \overline{s_{n-i}}p$ for each $i$, we choose
\[
\omega_1 =\frac{|s_i-\overline{s_{n-i}}p|}{s_i-\overline{s_{n-i}}p} \text{ and }
\omega_{2}=\sqrt{\frac{|s_{n-i}-\overline{s_i}p|}{s_{n-i}-\overline{s_i}p}}
\]
and get
\[
\tilde n_i(1-|p|^2)>
|s_i-\overline{s_{n-i}}p|+|s_{n-i}-\overline{s_i}p|\,.
\]
If for some $i$, $s_i= \overline{s_{n-i}}p$ or $s_{n-i}= \overline{s_{i}}p$
then the above inequality is obvious.\\

\noindent $(5)\Rightarrow (6).$ Since $|p|<1$ we choose
\[
c_i=\frac{s_i-\overline{s_{n-i}}p}{1-|p|^2} \quad \text{ for } i=1,\dots,n-1\,.
\]
It is evident that $s_i=c_i+\overline{c_{n-i}}p$.
Also by the hypothesis of $(5)$,
$Q=(c_1,c_2,\dots,c_{n-1})\in\mathbb G_{n-1}$ and hence $(6)$ follows.\\

\noindent $(3)\Leftrightarrow (4).$ The proof follows from
author's result Proposition 2.5 in \cite{sourav3}, only the
symbols $'\geq '$ and $' \leq '$ are to be replaced by $' > '$ and
$' < '$ respectively. Hence the proof is complete.

\end{proof}

\noindent In \cite{costara1}, Costara proved that the distinguished boundary $b\G$ of the symmetrized polydisc is the symmetrization of the $n$-torus $\T^n$, that is,
\[
b\G =\left\{ \left(\sum_{1\leq i\leq n} z_i,\sum_{1\leq
i<j\leq n}z_iz_j,\dots,
\prod_{i=1}^n z_i \right): \,|z_i|= 1, i=1,\dots,n \right \}.
\]
Note that $b\G$ is a proper subset of $\G \setminus \gn$, the topological boundary of $\gn$. The next theorem, which appeared in \cite{BSR}, provides
a set of characterizations for the points of the distinguished
boundary $b\Gamma_n$ of $\Gamma_n$.

\begin{thm}[\cite{BSR}, Theorem 2.4]\label{thm:DB}
For $(s_1,\dots,s_{n-1},p)\in\mathbb C^n$ the following are equivalent:
\begin{enumerate}
\item $(s_1,\dots,s_{n-1},p)\in b\Gamma_n$\,; \item $(s_1,\dots,s_{n-1},p)\in\Gamma_n$
and $|p|=1$\,; \item $|p|=1$, $s_i=\overline{s_{n-i}}p$ and
$\left(\dfrac{n-1}{n}s_1,\dfrac{n-2}{n}s_2,\dots, \dfrac{1}{n}s_{n-1}\right)\in\Gamma_{n-1}$\,; \item $|p|=1$
and there exists $(c_1,\dots,c_{n-1})\in b\Gamma_n$ such that\\
$ s_i=c_i+\overline{c_{n-i}}p $ , for $i=1,\dots,n-1$.
\end{enumerate}
\end{thm}

\noindent We now present a set of characterizations for the points in the closed symmetrized polydisc $\Gamma_n$.

\begin{thm}\label{thm:SC1}
Let $(s_1,\dots,s_{n-1},p)\in\mathbb C^n$ and let
\[
Q=\left( \dfrac{s_1-\overline{ s_{n-1}}p}{1-|p|^2}, \dfrac{s_2-\overline{ s_{n-2}}p}{1-|p|^2},\dots,\dfrac{s_{n-1}-\overline{ s_{1}}p}{1-|p|^2} \right)\in\mathbb C^{n-1}\,,
\]
when $|p|\neq 1$ and
\[
R=\left(\dfrac{n-1}{n}s_1,\dfrac{n-2}{n}s_2,\dots, \dfrac{1}{n}s_{n-1}\right).
\]
Then the following are equivalent:
\begin{enumerate}
\item $(s_1,\dots,s_{n-1},p)\in\Gamma_n$\,; \item $\Gamma_n$ is a
complete spectral set for $(s_1,\dots,s_{n-1},p)$ \,; \item
$(\omega s_1,\omega^2 s_2,\dots, \omega^{n-1}s_{n-1},\omega^n
p)\in\Gamma_n$ for all $\omega \in\mathbb T$  \,; \item
$\Phi_{i}(\alpha s_1,\dots, \alpha^{n-1} s_{n-1},\alpha^n p)\geq
0$ for all $\alpha\in\overline{\mathbb D}$, $i=1,2, \dots, n-1$
and either $|p|=1$ and $R \in\Gamma_{n-1}$ or $Q\in \Gamma_{n-1}$
\,; \item For each $i=1,\dots,n-1$,$\left| {n\alpha^n
p-\alpha^{n-i} s_{n-i}}\right|\leq |{n-\alpha^i s_i}|$, for all
$\alpha\in\overline{\mathbb D}$ and either $|p|=1$ and $R
\in\Gamma_{n-1}$ or $Q\in \Gamma_{n-1}$ \,; \item
$|s_i-\bar{s_{n-i}}p|+|s_{n-i}-\overline{s_i}p|\leq n(1-|p|^2)$
and either $|p|=1$ and $R \in\Gamma_{n-1}$ or $Q\in \Gamma_{n-1}$
\,; \item $|p|\leq 1$ and there exists
$(c_1,\dots,c_{n-1})\in\Gamma_{n-1}$ such that
\[
s_i=c_i+\overline{{c}_{n-i}}p\;, \quad i=1,\dots,n-1.
\]
\end{enumerate}

\end{thm}

\begin{proof}
The equivalence of $(1)$ and $(2)$ is obvious and follows from the definition.
Part $(1)\Leftrightarrow (3)$ was established by the author in Lemma 2.3 in \cite{sourav3}.
Also $(1)\Leftrightarrow (7)$ was proved by Costara, ( see Theorem 3.7 in \cite{costara1}).
The equivalence of $(4),(5)$ and $(6)$ follows from the previous result (Theorem \ref{thm:SC2}).
So it suffices if we prove $(1)\Rightarrow (4)$ and $(6)\Rightarrow (7)$.\\

\noindent $(1)\Rightarrow(4).$ Let $(s_1,\dots,s_{n-1},p)\in \Gamma_n$. Then the fact that
\[
\Phi_{i}(\alpha s_1,\dots,
\alpha^{n-1} s_{n-1},\alpha^n p)\geq 0\,,
\]
for all $\alpha\in\overline{\mathbb D}$, $i=1,2, \dots, n-1$, was
established in Proposition 2.5 in \cite{sourav3} by the author.
Since $(s_1,\dots,s_{n-1},p)\in \Gamma_n$, we have that $|p|\leq
1$. If $|p|=1$, then by Theorem \ref{thm:DB},
$(s_1,\dots,s_{n-1},p)\in b\Gamma_n$ and consequently $R\in
\Gamma_{n-1}$ by Theorem \ref{thm:DB}. Let $|p|<1$. Since
$(s_1,\dots,s_{n-1},p)\in \Gamma_n$ and $|p|<1$, by Costara's
result (Theorem 3.7 in \cite{costara1}), there exists a unique
$(c_1,\dots,c_{n-1})\in\Gamma_{n-1}$ such that
\[
s_i=c_i+\overline{c_{n-i}}p \quad \text{ for each } i=1,\dots,n-1.
\]
Again since for each $i=1,\dots,n-1$,
\[
s_i=\left( \frac{s_i-\overline{s_{n-i}p}}{1-|p|^2} \right) + \left( \overline{\frac{s_{n-i}-\bar s_i p}{1-|p|^2}} \right) p\,,
\]
we have that $Q=(c_1,\dots,c_{n-1})\in \Gamma_{n-1}$.\\

\noindent $(6)\Rightarrow (7).$ Let $(6)$ holds. Then $|p|\leq 1$. If $|p|=1$, then the left hand side of $(6)$ reduces to
\[
s_i=\overline{s_{n-i}}p \quad \text{ for each } i=1,\dots, n-1.
\]
Also since $R\in\Gamma_{n-1}$, by Theorem \ref{thm:DB},
$(s_1,\dots,s_{n-1},p)\in b\Gamma_n$. So, there exists
$(c_1,\dots,c_{n-1})\in b\Gamma_{n-1}$ such that
\[
s_i=c_i+\overline{c_{n-i}}p \quad \text{ for } i=1,\dots, n-1.
\]
When $|p|<1$, we can write
\begin{align*}
s_i & =\left( \frac{s_i-\overline{s_{n-i}p}}{1-|p|^2} \right) +
\left( \overline{\frac{s_{n-i}-\bar s_i p}{1-|p|^2}} \right) p
\\& = c_i+\overline{c_{n-i}}p\,,
\end{align*}
where $Q=(c_1,\dots,c_{n-1})\in\Gamma_{n-1}$. Therefore, $(7)$ is
established and the proof is complete.

\end{proof}

\vspace{0.1cm}

\section{The $\Gamma_n$-unitaries and $\Gamma_n$-isometries}\label{unitaries-isometries}

\vspace{0.4cm}

\noindent In this Section, we provide
several new characterizations for the $\Gamma_n$-unitaries and
$\Gamma_n$-isometries. Also, we make a refinement of
the model for pure $\Gamma_n$-isometries described in \cite{BSR}. We state a lemma first whose proof is a routine exercise.
\begin{lem} \label{basiclemma} Let $Q$ be a bounded operator on a Hilbert space
$\mathcal H$. If Re $\beta Q \le 0$ for all complex numbers
$\beta$ of modulus $1$, then $Q = 0$.
\end{lem}

\noindent We begin with a set of independent characterizations for a $\Gamma_n$-unitary.

\begin{thm} \label{G-unitary}
Let $(S_1,\dots,S_{n-1},P)$ be a commuting operator tuple defined
on a Hilbert space $\mathcal{H}.$ Then the following are
equivalent:
\begin{enumerate}

\item $(S_1,\dots,S_{n-1},P)$ is a $\Gamma_n$-unitary \, ; \item there exist
commuting unitary operators $U_{1},\dots,U_n$ on $\mathcal{H}$
such that
\[
(S_1,\dots,S_{n-1},P)= \pi_n(U_1,\dots,U_{n})\,;
\]
\item $P$ is unitary, $S_i=S_{n-i}^*P$ and
$\left(\dfrac{n-1}{n}S_1,\dots,\dfrac{1}{n}S_{n-1}\right)$ is a
$\Gamma_{n-1}$-contraction \,; \item $(S_1,\dots,S_{n-1},P)$ is a
$\Gamma_n$-contraction and $P$ is a unitary \,; \item $P$ is
unitary and there exists a $\Gamma_{n-1}$-unitary
$(C_1,\dots,C_{n-1})$ on $\mathcal H$ such that
$C_1,\dots,C_{n-1}$ commute with $P$ and
\[
S_i=C_i+C_{n-i}^*P \,, \quad i=1,\dots,n-1.
\]
\end{enumerate}
\end{thm}

\begin{proof}
The equivalence of conditions $(1),(2)$ and $(3)$ were established in \cite{BSR}. We shall show here:\\

$(2)\Rightarrow (4)\Rightarrow (3)$ and $(2)\Rightarrow (5)
\Rightarrow (3)$.\\

\noindent (2) $\Rightarrow$ (4) is trivial.\\

\noindent (4) $\Rightarrow$ (3). Let $(S_1,\dots,S_{n-1},P)$ be a
$\Gamma_n$-contraction and $P$ is a unitary. Since
$(S_1,\dots,S_{n-1},P)$ is a $\Gamma_n$-contraction, by
Proposition \ref{lem:3},
\[
\Phi_{i}(\omega S_1,\dots,\omega^{n-1} S_{n-1},\omega^n P)\geq 0 \text{ for all
} \omega \in \mathbb T \;,i=1,\dots,n-1.
\]

Now $\Phi_{i}(\omega S_1,\dots,\omega^{n-1} S_{n-1},\omega^n P)\geq 0$ gives
\[
\tilde n_i^2(I-P^*P)+(S_i^*S_i-S_{n-i}^*S_{n-i})- \tilde
n_i\omega(S_i-S_{n-i}^*P)- \tilde
n_i\bar{\omega}(S_i^*-P^*S_{n-i})\geq 0,
\]
which along with the fact that $P$ is a unitary implies that
\begin{equation}\label{13}
(S_i^*S_i-S_{n-i}^*S_{n-i})- \tilde n_i\omega(S_i-S_{n-i}^*P)-
\tilde n_i\bar{\omega}(S_i^*-P^*S_{n-i})\geq 0 \quad \forall
\omega \in \mathbb T.
\end{equation}
Putting $\omega =1$ and $-1$ respectively in (\ref{13}) and adding
them up we get
\begin{equation}\label{14}
S_i^*S_i-S_{n-i}^*S_{n-i} \geq 0.
\end{equation}
Since this holds for every $i$, replacing $i$ by $n-i$ we have
$
S_{n-i}^*S_{n-i}-S_i^*S_i \geq 0.
$
Hence $S_i^*S_i=S_{n-i}^*S_{n-i}$ for each $i=1,\dots, n-1$.
Therefore from (\ref{13}) we have that Re $\omega (S_i-S_{n-i}^*P)\leq
0$ for all $\omega\in\mathbb T$. By Lemma \ref{basiclemma},
$S_i=S_{n-i}^*P$. Again since $(S_1,\dots,S_{n-1},P)$ is a
$\Gamma_n$-contraction by Lemma \ref{lem:BSR2},
$
\left(\dfrac{n-1}{n}S_1,\dfrac{n-2}{n}S_2,\dots,\dfrac{1}{n}S_{n-1} \right)
$
is a $\Gamma_{n-1}$-contraction.\\

\noindent (2)$\Rightarrow$(5). Suppose (2) holds. Then for $i=1,\dots,n-1$,
$
S_i= \sum_{1\leq k_1 \leq k_2 \dots \leq k_i \leq n}
U_{k_1}\dots U_{k_i} \quad \text{ and } \quad P=\prod_{i=1}^{n}U_i\,.
$
By the equivalence of $(1)$ and $(2)$ we have that every
$\Gamma_n$-unitary is just the symmetrization of commuting $n$
unitaries. So let us consider the $\Gamma_{n-1}$-unitary
$
(C_1,\dots,C_{n-1})=\pi_{n-1}(U_1,\dots,U_{n-1}).
$
Needless to mention that $C_1,\dots,C_{n-1}$ commute with $P$.
Note that
\begingroup
\begin{align*}
S_1 &=&\sum_{i=1}^n U_i=\sum_{i=1}^{n-1} U_i + (U_1U_2\dots U_{n-1})^*P= C_1+C_{n-1}^*P \\
S_2 & = &\sum_{1\leq i<j\leq n} U_iU_j = \sum_{1\leq i<j\leq n-1}U_iU_j + \sum_{1\leq i\leq n-1}U_iU_n = C_2+C_{n-2}^*P \\
\vdots \\
S_{n-1} & = &\sum_{1\leq k_1 < k_2 < \dots < k_{n-1} \leq n}
U_{k_1}\dots U_{k_{n-1}} \\ & = & U_1U_2\dots U_{n-1}+
(\sum_{i=1}^{n-1} U_i)^*P =C_{n-1}+C_1^*P\,.
\end{align*}
\endgroup

\noindent (5)$\Rightarrow $(3) It suffices if we prove here that
$
\left(\dfrac{n-1}{n}S_1,\dfrac{n-1}{n}S_2,\dots,\dfrac{1}{n}S_{n-1} \right)
$
is a $\Gamma_{n-1}$-contraction, because, $S_i=S_{n-i}^*P$ is
obvious for every $i=1,\dots,n-1$. Now since $(C_1,\dots,C_{n-1})$
is a $\Gamma_{n-1}$-unitary, there are commuting unitaries
$U_1,\dots,U_{n-1}$ such that
$
(C_1,\dots,C_{n-1})=\pi_{n-1}(U_1,\dots,U_{n-1}).
$
We show that $(S_1,\dots,S_{n-1},P)$ is the symmetrization of the commuting unitaries $U_1,\dots ,U_{n-1}, C_{n-1}^*P$. Let $U_n=C_{n-1}^*P=U_1^*\dots U_{n-1}^*P$. So, we have
\begin{align*}
S_1 & = C_1+C_{n-1}^*P \\
& = \sum_{i=1}^{n-1}U_i\;+\; (U_1^*\dots U_{n-1}^*)P \\
& = \sum_{i=1}^{n}U_i
\\
S_2 & =  C_2 + C_{n-2}^*P \\
& = \sum_{1\leq i<j \leq n-1} U_iU_j\;+\; \left(\sum_{1\leq k_1<k_2<\dots <k_{n-2}} U_{k_1}^*\dots U_{k_{n-2}}^* \right)P \\
& = \sum_{1\leq i<j \leq n-1} U_iU_j\;+\; \left(\sum_{i=1}^{n-1}U_i \right)U_1^*\dots U_{n-1}^*P \\
& = \sum_{1\leq i<j \leq n}U_iU_j \\
\vdots \\
\vdots \\
S_{n-1} & = C_{n-1}+C_1^*P \\
&= U_1\dots U_{n-1} \;+\; \left( \sum_{i=1}^{n-1}U_i^* \right)P \\
& =  U_1\dots U_{n-1} \;+\; +\left( \sum_{1\leq k_1<k_2<\dots <k_{n-2}\leq n-1}U_{k_1}\dots U_{k_{n-2}}  \right) U_1^*U_2^*\dots U_{n-1}^*P \\
& = U_1\dots U_{n-1} \;+\; +\left( \sum_{1\leq k_1<k_2<\dots <k_{n-2}\leq n-1}U_{k_1}\dots U_{k_{n-2}}  \right) U_n \\
& = \sum_{1\leq k_1<k_2<\dots <k_{n-1}\leq n}U_{k_1}\dots U_{k_{n-1}}\,.
\end{align*}
Thus $(S_1,\dots, S_{n-1},P)$ is a $\Gamma_n$-unitary and hence a $\Gamma_n$-contraction. So, by Lemma \ref{lem:BSR2},
\[
\left(\dfrac{n-1}{n}S_1,\dfrac{n-1}{n}S_2,\dots,\dfrac{1}{n}S_{n-1} \right)
\]
is a $\Gamma_{n-1}$-contraction. Hence the proof is complete.

\end{proof}

\subsection{A set of characterizations for the $\Gamma_n$-isometries}

\noindent We begin this subsection with a lemma whose proof is
straight-forward and could be found in \cite{tirtha-sourav}.

\begin{lem}\label{123}
Let $U\;,\;V$ be a unitary and a pure isometry on Hilbert Spaces
$\mathcal{H}_1 \;,\; \mathcal{H}_2$ respectively, and let
$X\;:\;\mathcal{H}_1 \rightarrow \mathcal{H}_2$ be such that
$XU=VX.$ Then $X=0.$
 \end{lem}

\noindent Now we present a set of characterizations for the
$\Gamma_n$-isometries, few parts of which appeared in \cite{BSR}.

\begin{thm}\label{G-isometry}
Let $S_1,\dots,S_{n-1},P$ be commuting operators on a Hilbert space
$\mathcal H$. Then the following are equivalent:
\begin{enumerate}

\item $(S_1,\dots,S_{n-1},P)$ is a $\Gamma_n$-isometry ; \item $P$
is isometry, $S_i=S_{n-i}^*P$ for each $i=1,\dots,n-1$ and
$\left(\dfrac{n-1}{n}S_1,\dfrac{n-2}{n}S_2,\dots
,\dfrac{1}{n}S_{n-1} \right)$ is a $\Gamma_{n-1}$-contraction ;
\item $($ Wold-Decomposition $)$: there is an orthogonal
decomposition $\mathcal H=\mathcal H_1 \oplus \mathcal H_2$ into
common invariant subspaces of $S_1,\dots,S_{n-1}$ and $P$ such
that $(S_1|_{\mathcal H_1},\dots,S_{n-1}|_{\mathcal
H_1},P|_{\mathcal H_1})$ is a $\Gamma_n$-unitary and
$(S_1|_{\mathcal H_2},\dots,S_{n-1}|_{\mathcal H_2},P|_{\mathcal
H_2})$ is a pure $\Gamma_n$-isometry ; \item
$(S_1,\dots,S_{n-1},P)$ is a $\Gamma_n$-contraction and $P$ is an
isometry; \item $\left(\dfrac{n-1}{n}S_1,\dfrac{n-2}{n}S_2,\dots
,\dfrac{1}{n}S_{n-1} \right)$ is a $\Gamma_{n-1}$-contraction and
\[
\rho_{i}(\omega S_1,\dots,\omega^{n-1}S_{n-1},\omega^nP)= 0,
\]
for all $\omega\in \mathbb T$ and $\forall \; i=1,\dots,n-1$;\\

\noindent Moreover, if the spectral radius $r(S_i)$ is less than
$\tilde n_i \,(= \binom{n}{i})$ for each $i=1,\dots,n-1$ then all
of the above are equivalent to :

\item $\left(\dfrac{n-1}{n}S_1,\dots,\dfrac{1}{n}S_{n-1}\right)$
is a $\Gamma_{n-1}$-contraction and for each $i=1,\dots, n-1$,
$(\tilde n_i\beta P-S_{n-i})(\tilde n_i-\beta S_i)^{-1}$ is an
isometry for all $\beta\in\mathbb T$.

\end{enumerate}

\end{thm}
\begin{proof}
The equivalence of $(1),(2),(3)$ and $(5)$ was shown in Theorem
4.12 in \cite{BSR}. We prove here $(1)\Rightarrow (4)\Rightarrow
(5)$ and when $r(S_i)< \tilde n_i$ for all $i$, then
$(5)\Leftrightarrow (6)$.\\

\noindent $(1)\Rightarrow (4)$ Since $(S_1,\dots,S_{n-1},P)$ is a
$\Gamma_n$-isometry, it is the restriction of a $\Gamma_n$-unitary
say $(\tilde S_1,\dots, \tilde S_{n-1}, \tilde P)$ to a common
invariant subspace of $\tilde S_1,\dots, \tilde S_{n-1}$ and
$\tilde P$. Also since $\tilde P$ is a unitary, its restriction to
an invariant subspace is an isometry. Therefore, $P$ is an
isometry. Again $(S_1,\dots,S_{n-1},P)$, being the restriction of
the $\Gamma_n$-contraction $(\tilde S_1,\dots, \tilde S_{n-1},
\tilde P)$ to a common invariant subspace of $\tilde S_1,\dots,
\tilde S_{n-1}, \tilde P $, is a $\Gamma_n$-contraction.\\

\noindent $(4)\Rightarrow (5)$ Since $(S_1,\dots,S_{n-1},P)$ is a
$\Gamma_n$-contraction, by Lemma \ref{lem:BSR2},
\[
\left(\dfrac{n-1}{n}S_1,\dfrac{n-2}{n}S_2,\dots ,\dfrac{1}{n}S_{n-1} \right)
\]
is a $\Gamma_{n-1}$-contraction. Again since
$(S_1,\dots,S_{n-1},P)$ is a $\Gamma_n$-contraction, by
Proposition \ref{lem:3},
\[
\Phi_{i}(\beta S_1,\dots, {\beta}^{n-1}S_{n-1}, {\beta}^nP)\geq 0\,,
\]
for $i=1,\dots,n-1$, and for all $\beta\in\mathbb T$.
So, we have
\[
\tilde n_i^2(I-P^*P)+(S_i^*S_i-S_{n-i}^*S_{n-i})-2 \tilde n_i
\text{ Re } \beta (S_i-S_{n-i}^*P) \geq 0.
\]
Using the fact that $P^*P=I$ we have
\[
(S_i^*S_i-S_{n-i}^*S_{n-i}) - \text{Re } \beta^i (S_1-S_2^*P)\geq 0 \text{ for all }
\beta\in\mathbb T.
\]
Choosing $\beta^i=1$ and $-1$ respectively we obtain from the above inequality
\[
(S_i^*S_i-S_{n-i}^*S_{n-i}) \geq 0.
\]
Similarly considering the positivity of the operator pencil $\Phi_{n-i}$ and repeating the same procedure we obtain
\[
(S_i^*S_i-S_{n-i}^*S_{n-i}) \leq 0.
\]
Therefore $S_i^*S_i=S_{n-i}^*S_{n-i}$ and hence
\[
\text{Re } \beta^i (S_1-S_2^*P) \leq 0 \quad \text{ for all } \beta \in\mathbb T.
\]
We apply \ref{basiclemma} to get $S_i=S_{n-i}^*P$ for $i=1,\dots,n-1$.
Accumulating all these facts together we conclude that
\[
\Phi_{i}(\beta S_1,\dots, {\beta}^{n-1}S_{n-1}, {\beta}^nP) =0\,,
\]
for all $\beta$ with unit modulus and $i=1,\dots,n-1$.
\\

\noindent $(5)\Rightarrow (6)$ For every $i=1,\dots,n-1$ and for all $\beta\in\mathbb T$, we have
\begin{align*}
& \Phi_{i}(\beta S_1,\dots, {\beta}^{n-1}S_{n-1}, {\beta}^nP)\\
& = (\tilde n_i -\beta^i S_i)^*(\tilde n_i - \beta^i S_i)-(\tilde
n_i{\beta}^nP-{\beta}^{n-i}S_{n-i})^*(\tilde
n_i{\beta}^nP-{\beta}^{n-i}S_{n-i})
\\& =0\,.
\end{align*}
This implies that
\[
(\tilde n_i-\beta^i S_i)^*(\tilde n_i-\beta^i S_i)=(\tilde
n_i{\beta}^nP-{\beta}^{n-i}S_{n-i})^*(\tilde
n_i{\beta}^nP-{\beta}^{n-i}S_{n-i}).
\]
Since $r(S_i)< \tilde n_i$ , $\tilde n_i-\beta^i S_i$ is
invertible and so we have
\[
((\tilde n_i-\beta^i S_i)^{-1})^*(\tilde
n_i{\beta}^nP-{\beta}^{n-i}S_{n-i})^*(\tilde
n_i{\beta}^nP-{\beta}^{n-i}S_{n-i})(\tilde n_i-\beta^i
S_i)^{-1}=I.
\]
Therefore, $(\tilde n_i{\beta}^nP-{\beta}^{n-i}S_{n-i})(\tilde
n_i-\beta^i S_i)^{-1}$ is an isometry and hence $(\tilde
n_i{\beta}^iP-S_{n-i})(\tilde n_i-\beta^i S_i)^{-1}$ is an
isometry. This is same as saying that $(\tilde
n_i{\nu}P-S_{n-i})(\tilde n_i-\nu S_i)^{-1}$ is an isometry for
all
$\nu\in\mathbb T$ and for all $i=1,\dots,n-1$.\\

Conversely, let $(5)$ holds. Then $(\tilde n_i{\beta}P-S_2)(\tilde
n_i-\beta S_i)^{-1}$ is an isometry for all $\beta\in\mathbb T$
and for each $i=1,\dots,n-1$. Therefore,
\[
((\tilde n_i-\beta^i S_i)^{-1})^*(\tilde
n_i{\beta}^nP-{\beta}^{n-i}S_{n-i})^*(\tilde
n_i{\beta}^nP-{\beta}^{n-i}S_{n-i})(\tilde n_i-\beta^i S_i)^{-1}=I
\]
or equivalently for all $\beta\in\mathbb T$,
\begin{align*}
& \Phi_{i}(\beta S_1,\dots,{\beta}^{n-1}S_{n-1},{\beta}^nP) \\
& = (\tilde n_i-\beta^i
S_i)^*(\tilde n_i-\beta^{i}S_i)-(\tilde n_i{\beta}^nP-{\beta}^{n-i}S_{n-i})^*(\tilde n_i{\beta}^nP-{\beta}^{n-i}S_{n-i}) \\
& =0.
\end{align*}
This completes the proof.
\end{proof}

\subsection{A model theorem for pure $\Gamma_n$-isometry}

Wold-decomposition splits an isometry into two
orthogonal parts namely, a unitary and a pure isometry (\cite{nagy}, Chapter-I). A pure isometry $V$ is unitarily equivalent to the Toeplitz operator $T_z$ on the vectorial Hardy space $H^2(\mathcal D_{V^*})$. We have in Theorem \ref{G-isometry}, an analogous Wold-decomposition for $\Gamma_n$-isometries in terms of
a $\Gamma_n$-unitary and a pure $\Gamma_n$-isometry. Again Theorem \ref{G-unitary} shows that every $\Gamma_n$-unitary is nothing but the symmetrization of commuting $n$ unitaries. Therefore a model for pure $\Gamma_n$-isometries gives a complete picture of a $\Gamma_n$-isometry. In \cite{BSR}, Biswas and Shyam
Roy have provided such a model in terms of Toeplitz operator tuple $(T_{\phi_1},\dots, T_{\phi_{n-1}},T_z)$ on an abstract Hardy-Hilbert space (see Theorem 4.10 in \cite{BSR}). Here we make a refinement of their result by specifying the Hilbert space and the corresponding Toeplitz operators.

\begin{thm}\label{model1}
Let $(\hat{S_1},\dots,\hat{S}_{n-1},\hat{P})$ be a commuting
triple of operators on a Hilbert space $\mathcal H$. If
$(\hat{S_1},\dots,\hat{S}_{n-1},\hat{P})$ is a pure
$\Gamma_n$-isometry then there is a unitary operator $U:\mathcal H
\rightarrow H^2(\mathcal D_{{\hat{P}}^*})$ such that
\[
\hat{S}_i=U^*T_{\varphi_i}U, \text{ for } i=1,\dots,n-1\,, \text{
and }  \hat{P}=U^*T_zU\,.
\]
Here each $T_{\varphi_i}$ is the Toeplitz operator on the
vectorial Hardy space $H^2(\mathcal D_{{\hat{P}}^*})$ with the
symbol $\varphi_i(z)= F_i^*+F_{n-i}z$, where $(F_1,\dots,F_{n-1})$
is the fundamental operator tuple of
$(\hat{S}_1^*,\dots,\hat{S}_{n-1}^*,\hat{P}^*)$ such that
$$\left(
\frac{n-1}{n}(F_1^*+F_{n-1}z),\frac{n-2}{n}(F_2^*+F_{n-2}z),\dots,
\frac{1}{n}(F_{n-1}^*+F_1z) \right)$$ is a
$\Gamma_{n-1}$-contraction for every $z\in \mathbb T$.\\

Conversely, if $F_1,\dots,F_{n-1}$ are bounded operators on a
Hilbert space $E$ such that $$\left(
\frac{n-1}{n}(F_1^*+F_{n-1}z),\frac{n-2}{n}(F_2^*+F_{n-2}z),\dots,
\dfrac{1}{n}(F_{n-1}^*+F_1z) \right)$$ is a
$\Gamma_{n-1}$-contraction for every $z\in \mathbb T$, then
$(T_{F_1^*+F_{n-1}z},\dots,T_{F_{n-1}^*+F_1z},T_z)$ on $H^2(E)$ is
a pure $\Gamma_n$-isometry.
\end{thm}

\begin{proof}

Suppose that $(\hat{S}_1,\dots,\hat{S}_{n-1},\hat{P})$ is a pure
$\Gamma_n$-isometry. Then $\hat{P}$ is a pure isometry and it can
be identified with the Toeplitz operator $T_z$ on $H^2(\mathcal
D_{{\hat{P}}^*})$. Therefore, there is a unitary $U$ from
$\mathcal H$ onto $H^2(\mathcal D_{{\hat{P}}^*})$ such that
$\hat{P}=U^*T_zU$. Since for $\hat{S}_i$ is a commutant of
$\hat{P}$, there are multipliers $\phi_i$ in $H^{\infty}(\mathcal
B(\mathcal D_{{\hat{P}}^*}))$ such that $\hat{S}_i=U^*T_{\phi_i}U$
for $i=1,\dots, n-1$.

Since $(T_{\phi_1},\dots,T_{\phi_{n-1}},T_z)$ is a
$\Gamma_n$-isometry, by part-(3) of Theorem \ref{G-isometry}, we have
that
\[
T_{\phi_i}=T_{\phi_{n-i}}^*T_z \text{ for } i=1,\dots,n-1\,,
\]
and by these relations we have
\[
\phi_i(z)=G_i+G_{n-i}z \textup{ and }
\phi_{n-i}(z)=G_{n-i}^*+G_i^*z \textup{ for some } G_i,G_{n-i}
\in\mathcal B(\mathcal D_{{\hat{P}}^*}).
\]
Set ${F_i}= G_i^*$ and ${F_{n-i}}=G_{n-i}$ for each $i$. Again
since $(T_{\phi_1},\dots,T_{\phi_{n-1}},T_z)$ is a
$\Gamma_n$-isometry, by Lemma \ref{lem:BSR2}, $
(\frac{n-1}{n}T_{\phi_1},\dots,\frac{1}{n}T_{\phi_{n-1}}) $ is a
$\Gamma_{n-1}$-contraction and hence
$(\frac{n-1}{n}\phi_1,\dots,\frac{1}{n}\phi_{n-1})$, that is,
\[
\left(
\frac{n-1}{n}(F_1^*+F_{n-1}z),\frac{n-2}{n}(F_2^*+F_{n-2}z),\dots,
\frac{1}{n}(F_{n-1}^*+F_1z)
\right)
\]
is a $\Gamma_{n-1}$-contraction for all $z$ in $\mathbb T$.\\

For the converse part, note that
\[
\left(
\frac{n-1}{n}(F_1^*+F_{n-1}z),\frac{n-2}{n}(F_2^*+F_{n-2}z),\dots,
\frac{1}{n}(F_{n-1}^*+F_1z)
\right)
\]
is a $\Gamma_{n-1}$-contraction. So, the fact that
$(T_{F_1^*+F_{n-1}z},\dots,T_{F_{n-1}^*+F_1z},T_z)$ on $H^2(E)$ is
a $\Gamma_n$-isometry follows from part-(2) of Theorem
\ref{G-isometry}. Needless to mention that $T_z$ on $H^2(E)$ is a
pure isometry which makes
$(T_{F_1^*+F_{n-1}z},\dots,T_{F_{n-1}^*+F_1z},T_z)$ a pure
$\Gamma_n$-isometry.

\end{proof}

\begin{rem}\label{partial-converse}
Consider the adjoint of the pure $\G$-isometry $(T_{F_1^*+F_{n-1}z},\dots,T_{F_{n-1}^*+F_1z},T_z)$. It can be easily verified that $F_1,\dots, F_{n-1}$ satisfy the operator equations
\[
T_{F_i^*+F_{n-i}z}^*-T_{F_{n-i}^*+F_{i}z}T_z^*=D_{T_z^*}F_iD_{T_z^*} \quad 1\leq i \leq n-1.
\]
Thus for given operators $F_1,\dots, F_{n-1}$ on a Hilbert space on $E$ if
\[
\left(
\frac{n-1}{n}(F_1^*+F_{n-1}z),\frac{n-2}{n}(F_2^*+F_{n-2}z),\dots,
\frac{1}{n}(F_{n-1}^*+F_1z) \right)
\]
is a $\Gamma_{n-1}$-contraction, then $(F_1, \dots , F_{n-1})$ is the $\ft$-tuple of a $\G$-contraction which can be chosen to be the $\Gamma_n$-co-isometry $(T_{F_1^*+F_{n-1}z}^*,\dots,T_{F_{n-1}^*+F_1z}^*,T_z^*)$ acting on $H^2(E)$.

\end{rem}

\vspace{0.4cm}

\vspace{0.4cm}

\section{An explicit $\Gamma_n$-unitary dilation on the minimal dilation space}\label{conditional-dilation}

\vspace{0.4cm}

\noindent We have already mentioned that Sz.-Nagy's unitary dilation of a contraction was a pioneering work in the field of operator theory. This has deep interrelation with the commutant lifting and interpolations, see e.g. \cite{F-F1}. Moreover, when such a dilation is minimal, it provides a commutant lifting on a minimal space. Interestingly if a unitary dilation of a contraction is minimal, then it is unique in the sense that any two minimal unitary dilations of a contraction are unitarily equivalent. Thus, without loss of generality one can choose anyone of the minimal dilations of a contraction. Of the minimal unitary dilations the most appealing one is the Schaeffer's unitary dilation of a contraction (see \cite{nagy}) which gives an explicit construction of a minimal unitary dilation of a contraction. For a contraction $P$ on $\HS$, Schaeffer chose the space $\mathcal K =l^2(\mathcal D_P) \oplus \HS \oplus \mathcal D_{P^*}$ and the unitary operator as in (\ref{2.33}) acting on $\mathcal K$ to be the minimal unitary dilation of a $P$. Constructing a minimal dilation in multivariable setting has always been a challenging task. In \cite{tirtha-sourav}, the author and his collaborators have constructed a minimal $\Gamma_2$-unitary dilation for a $\Gamma_2$-contraction. In this Section, we shall construct $\G$-unitary dilation for a certain class of a $\G$-contraction $(S_1, \dots , S_{n-1},P)$. There are three notable facts about this $\G$-unitary dilation:
\begin{itemize}
\item[(i)] the dilation will act on the minimal unitary dilation space $\mathcal K$ of $P$;

\item[(ii)] it gives a complete characterization of the class of $\G$-contractions which admit a $\G$-unitary dilation on $\mathcal K$;

\item[(iii)] it generalizes the $\Gamma_2$-unitary dilation of a $\Gamma_2$-contraction obtained in \cite{sourav}.
\end{itemize}
Recall that a rational dilation on $\G$ is nothing but a $\G$-unitary dilation of a $\G$-contraction. By virtue of polynomial convexity of $\G$, it suffices to restrict the functional calculus to the algebra of polynomials, more specifically to the monomials.
\begin{defn} Let $(S_1,\dots,S_{n-1},P)$ be a
$\Gamma_n$-contraction on $\mathcal H$. A commuting tuple of operators $(R_1,\dots,R_{n-1},U)$ acting on $\widetilde{\HS}$ is said to be a
$\Gamma_n$-unitary dilation (or a $\Gamma_n$-isometric dilation) of $(S_1,\dots,S_{n-1},P)$ if $\mathcal H \subseteq \widetilde{\HS}$, $(R_1,\dots,R_{n-1},U)$ is a $\Gamma_n$-unitary (or a $\Gamma_n$-isometry) and
\[
P_{\mathcal H}(R_1^{m_1}\dots R_{n-1}^{m_{n-1}}U^{n})|_{\mathcal
H} =S_1^{m_1}\dots S_{n-1}^{m_{n-1}}P^{n},
\]
for all non-negative integers $ m_1, \dots, m_{n-1},n.$ Moreover, such a dilation is called {\em minimal} if
\[
\widetilde{\HS}=\overline{\textup{span}}\{ R_1^{m_1}\dots
R_{n-1}^{m_{n-1}}U^n h\,:\; h\in\mathcal H \textup{ and
}m_1,\dots, m_{n-1},n\in \mathbb N \cup \{0\} \}.
\]
\end{defn}
Throughout this section, we shall use the well-known fact that for any contraction $P$,
\begin{equation}\label{nagy-foias}
PD_P=D_{P^*}P.
\end{equation}
A proof to this could be found in \cite{nagy}. Also, we shall apply frequently the following operator identities
\begin{align}\label{funda-repeat}
& S_i-S_{n-i}^*P= D_{P}F_iD_{P}\,,
\\& \label{funda-repeat1} S_i^*-S_{n-i}P^*=
D_{P^*}{F_i}_*D_{P^*}, \qquad (1 \leq i \leq n-1)
\end{align}
whence $(A_1,\dots,A_{n-1})$ and $(B_1,\dots,B_{n-1})$ are the $\mathcal F$-tuples of $(S_1,\dots,S_{n-1},P)$ and its adjoint respectively. Before going to the main dilation theorem we present a lemma which will be useful.

\begin{lem}\label{funda-properties}
Let $(S_1,\dots,S_{n-1},P)$ be a $\Gamma_n$-contraction on a
Hilbert space $\mathcal H$ and let $(A_1,\dots,A_{n-1})$ and
$(B_1,\dots, B_{n-1})$ be the $\mathcal F$-tuples of
$(S_1,\dots,S_{n-1},P)$ and $(S_1^*,\dots,S_{n-1}^*,P^*)$ respectively. Then

\begin{enumerate}
\item $PA_i={B_i}^*P|_{\mathcal D_{P}}$ and
$P^*B_i={A_i}^*P^*|_{\mathcal D_{P^*}}$ for $i=1,\dots, n-1$;

\item $D_{P}S_i=A_iD_{P}+A_{n-i}^*D_{P}P$ for $i=1,\dots, n-1$;

\item $D_PA_i=(S_iD_P-D_{P^*}B_{n-i}P)|{\mathcal{D}_P}$ and $D_{P^*}B_i=(S_i^*D_{P^*}-D_{P}A_{n-i}P^*)|{\mathcal{D}_{P^*}}$, for $i=1,\dots, n-1$;

\item $S_iD_{P^*}=D_{P^*}B_i^*+PD_{P^*}{B_{n-i}}$ for $i=1,\dots,
n-1$;

\item $S_i^*S_j-S_{n-j}^*S_{n-i}=D_{P}(A_i^*A_j-A_{n-j}^*A_{n-i})D_{P}$, for any $i,j\in \{1,\dots, n-1 \}$, when $[A_j,A_{n-i}]=0$; 

\item ${S_i}S_j^*-S_{n-j}S_{n-i}^*=D_{P^*}(B_i^*B_j-B_{n-j}^*B_{n-i})D_{P^*}$, for any $i,j\in \{1,\dots, n-1 \}$, when $[{B_j},{B_{n-i}}]=0$.
\end{enumerate}
\end{lem}

\begin{proof} The proof is technical and is given in the 'Appendix'.

\end{proof}

\noindent Now we are in a position to present the construction of the desired $\Gamma_n$-unitary dilation which is one of the main results of this paper.

\begin{thm}\label{main-dilation-theorem}
 Let $(S_1,\dots ,S_{n-1},P)$ be a $\Gamma_n$-contraction defined on a Hilbert space $\mathcal
 H$ with $(A_1,\dots,A_{n-1})$ and $(B_{1},\dots,B_{{n-1}})$
 being the $\ft$-tuples of $(S_1,\dots,S_{n-1},P)$ and $(S_1^*,\dots ,S_{n-1}^*,P^*)$ respectively. Let $\mathcal K$ be the minimal unitary dilation space of $P$. Then $(S_1, \dots , S_{n-1},P)$ possesses a $\Gamma_n$-unitary dilation $(R_1, \dots, R_{n-1},U)$ on $\mathcal K$ with $U$ being the minimal unitary dilation of $P$ if and only if
 \[
 \left(
\frac{n-1}{n}(A_1+A_{n-1}^*z),\frac{n-2}{n}(A_2+A_{n-2}^*z),\dots,
\dfrac{1}{n}(A_{n-1}+A_1^*z) \right)
\]
and
\[
\left(
\frac{n-1}{n}(B_{1}^*+B_{n-1}z),\frac{n-2}{n}(B_{2}^*+B_{n-2}z),\dots,
\dfrac{1}{n}(B_{n-1}^*+B_{1}z) \right)
\]
are $\Gamma_{n-1}$-contraction for every $z\in \T$. Moreover, this $\Gamma_n$-unitary dilation on $\mathcal K$ is unique and minimal.

\end{thm}

\begin{proof}

\textbf{\textit{The} $( \Leftarrow )$ \textit{part}.}\\

\noindent Suppose that
\[
\Sigma_1 = \left(
\frac{n-1}{n}(A_1+A_{n-1}^*z),\frac{n-2}{n}(A_2+A_{n-2}^*z),\dots,
\dfrac{1}{n}(A_{n-1}+A_1^*z) \right)
\]
and
\[
\Sigma_2= \left(
\frac{n-1}{n}(B_{1}^*+B_{n-1}z),\frac{n-2}{n}(B_{2}^*+B_{n-2}z),\dots,
\dfrac{1}{n}(B_{n-1}^*+B_{1}z) \right)
\]
are $\Gamma_{n-1}$-contraction for every $z\in \T$. We show an explicit construction of a $\Gamma_n$-unitary on $\mathcal K$ that dilates $(S_1, \dots , S_{n-1},P)$. Since the minimal unitary dilation space of a contraction is unique upto a unitary (see Chapter-I of \cite{nagy}), without loss of generality let us choose $\mathcal K$ to be the minimal Schaeffer unitary dilation space, that is,
\[
\mathcal K=\cdots\oplus\mathcal D_{P}\oplus\mathcal D_{P}\oplus\mathcal D_{P}\oplus\mathcal H\oplus
 \mathcal D_{P^*}\oplus\mathcal D_{P^*}\oplus\mathcal D_{P^*}\oplus\cdots =l^2(\mathcal D_P)\oplus \mathcal H \oplus l^2(\mathcal D_{P^*}).
\]
Let us define $(R_1,\dots,R_{n-1},U)$ on $\mathcal K$ by
\begin{eqnarray}\label{2.3}
&R_i =\footnotesize \left[
\begin{array}{ c c c c|c|c c c c}
\bm{\ddots}&\vdots &\vdots&\vdots   &\vdots  &\vdots& \vdots&\vdots&\vdots\\
\cdots&0&A_i&A_{n-i}^*  &0&  0&0&0&\cdots\\
\cdots&0&0&A_i  &A_{n-i}^*D_{P}&  -A_{n-i}^*P^*&0&0&\cdots\\
\hline

\cdots&0&0&0   &S_i&   D_{P^*}{B_{n-i}}&0&0&\cdots\\ \hline

\cdots&0&0&0   &0&  {B_i}^*& {B_{n-i}}&0&\cdots\\
\cdots&0&0&0   &0&  0&{B_i}^*&{B_{n-i}}&\cdots\\
\vdots&\vdots&\vdots&\vdots&\vdots&\vdots&\vdots&\vdots&\bm{\ddots}\\
\end{array} \right]\,,\\&
\notag 1\leq i \leq n-1\,,
\\&
\text{ \large and }\quad U = \left[
\begin{array}{ c c c c|c|c c c c}
\bm{\ddots}&\vdots &\vdots&\vdots   &\vdots  &\vdots& \vdots&\vdots&\vdots\\
\cdots&0&0&I  &0&  0&0&0&\cdots\\
\cdots&0&0&0  &D_{P}&  -{P}^*&0&0&\cdots\\ \hline

\cdots&0&0&0   &P&   D_{P^*}&0&0&\cdots\\ \hline

\cdots&0&0&0   &0&  0& I&0&\cdots\\
\cdots&0&0&0   &0&  0&0&I&\cdots\\
\vdots&\vdots&\vdots&\vdots&\vdots&\vdots&\vdots&\vdots&\bm{\ddots} \label{2.33}\\
\end{array} \right].
\end{eqnarray}
Note that we have chosen $U$ to be the minimal Schaeffer unitary dilation of $P$. We show that $(R_1, \dots, R_{n-1},U)$ is a $\Gamma_n$-unitary dilation of $(S_1, \dots , S_{n-1},P)$. Suppose the block matrix of $R_i$ for $i=1, \dots , n-1$ and $U$ with respect to the decomposition $l^2(\mathcal
D_{P})\oplus \mathcal H \oplus l^2(\mathcal D_{P^*})$ of $\mathcal
K$ are
\begin{equation}\label{sec:eqn-011}
\left[
\begin{array}{ccc}
G_{i1} & G_{i2} & G_{i3}\\
0 & S_i & G_{i4}\\
0& 0& G_{i5}  \end{array} \right] \text{ and } \left[
\begin{array}{ccc}
Q_{1} & Q_{2} & Q_{3}\\
0 & P & Q_{4}\\
0& 0& Q_{5}  \end{array} \right]
\end{equation}
respectively. If $R_1,\dots,R_{n-1},U$ commute, then it is obvious
from the upper triangular forms of the block matrices of
$R_1,\dots,R_{n-1},U$ that
\[
 P_{\mathcal H}(R_1^{m_1}\dots
R_{n-1}^{m_{n-1}}U^m)|_ {\mathcal H}=S_1^{m_1}\dots
S_{n-1}^{m_{n-1}}P^m\,,
\]
for all integers $m_1,\dots, m_{n-1},m$. This proves that
$(R_1,\dots, R_{n-1},U)$ dilates $(S_1,\dots,S_{n-1},P)$. The
minimality of the dilation follows from the fact that $\mathcal K$
and $U$ are respectively the minimal unitary dilation space and
minimal unitary dilation of $P$. Therefore, in order to prove that
$(R_1,\dots,R_{n-1},U)$ is a minimal $\Gamma_n$-unitary dilation
of $(S_1,\dots,S_{n-1},P)$, all we need to show is that
$(R_1,\dots,R_{n-1},U)$ is a $\Gamma_n$-unitary. By Theorem
\ref{G-unitary}, it is enough if we show that
$(R_1,\dots,R_{n-1},U)$ is a $\Gamma_n$-contraction because $U$ is
unitary. Again by virtue of polynomial convexity of $\Gamma_n$, it
suffices to prove

\begin{enumerate}
\item $R_1,\dots , R_{n-1},U$ commute and \item $\| f(R_1,\dots ,
R_{n-1},U) \| \leq \| f \|_{\infty , \Gamma_n}$ ,  for all $f \in
\mathbb C[z_1,\dots, z_n].$
\end{enumerate}

By the commutativity of the tuples $\Sigma_1$ and $\Sigma_2$ for all $z\in \T$, we have that
\begin{align}\label{rel:01}
& (1)\; A_iA_j=A_jA_i \;, \quad B_{i}B_{j}=B_{j}B_{i}, \notag \\
&(2)\; [A_i^*,A_{n-j}]=[A_j^*,A_{n-i}] \;, \quad
[B_{i}^*,B_{{n-j}}]=[B_{j}^*,B_{{n-i}}].
\end{align}

\noindent \textit{Step 1.} Note that $R_iR_j =$
\begin{equation*}
\tiny \left[
\begin{array}{ c c c|c|c c c }
\bm{\ddots} &\vdots&\vdots   &\vdots  &\vdots& \vdots&\vdots\\
\cdots & A_iA_j& A_iA_{n-j}^*+A_{n-i}^*A_j &A_{n-i}^*A_{n-j}^*D_{P}&  -A_{n-i}^*A_{n-j}^*P&0&\cdots\\
\cdots&0&A_iA_j  & {\normalsize \substack{A_iA_{n-j}^*D_{P}\\
+A_{n-i}^*D_{P}S_j}}&
{\normalsize \substack{-A_iA_{n-j}^*P^*+A_{n-i}^*D_{P}D_{P^*}{B_{n-j}}\\
-A_{n-i}^*P^*{B_j}^*}} &-A_{n-i}^*P^*{B_{n-j}}&\cdots\\
\hline

\cdots&0&0 &S_iS_j&
S_iD_{P^*}{B_{n-j}}+D_{P^*}{B_{n-i}}{B_j}^*&D_{P^*}{B_{n-i}}{B_{n-j}}&\cdots\\
\hline

\cdots&0&0 &0&  {B_i}^*{B_j}^*& {B_i}^*{B_{n-j}}+{B_{n-i}}{B_j}^*&\cdots\\
\cdots &0&0 &0&  0&{B_i}^*{B_j}^*& \cdots\\
\vdots &\vdots&\vdots&\vdots&\vdots&\vdots &\bm{\ddots}\\
\end{array} \right]
\end{equation*}
and $R_jR_i =$
\begin{equation*}
\tiny \left[
\begin{array}{ c c c|c|c c c }
\bm{\ddots} &\vdots&\vdots   &\vdots  &\vdots& \vdots&\vdots\\
\cdots & A_jA_i& A_jA_{n-i}^*+A_{n-j}^*A_i &A_{n-j}^*A_{n-i}^*D_{P}&  -A_{n-j}^*A_{n-i}^*P&0&\cdots\\
\cdots&0&A_jA_i  & {\normalsize \substack{A_jA_{n-i}^*D_{P}\\
+A_{n-j}^*D_{P}S_i}}&
{\normalsize \substack{-A_jA_{n-i}^*P^*+A_{n-j}^*D_{P}D_{P^*}{B_{n-i}}\\
-A_{n-j}^*P^*{B_i}^*}} &-A_{n-j}^*P^*{B_{n-i}}&\cdots\\
\hline

\cdots&0&0 &S_jS_i&
S_jD_{P^*}{B_{n-i}}+D_{P^*}{B_{n-j}}{B_i}^*&D_{P^*}{B_{n-j}}{B_{n-i}}&\cdots\\
\hline

\cdots&0&0 &0&  {B_j}^*{B_i}^*& {B_j}^*{B_{n-i}}+{B_{n-j}}{B_i}^*&\cdots\\
\cdots &0&0 &0&  0&{B_j}^*{B_i}^*& \cdots\\
\vdots &\vdots&\vdots&\vdots&\vdots&\vdots &\bm{\ddots}\\
\end{array} \right]
\end{equation*}
It is evident from relations (\ref{rel:01}) that except for the
entities $(-1,0), (0,1), (-1,1)$ in the matrices of $R_iR_j$ and
$R_jR_i$, all other corresponding entities are equal. So, for
proving $R_iR_j=R_jR_i$, it is enough to verify that the entities
$(-1,0), (0,1), (-1,1)$ of $R_iR_j$ are equal to that of $R_jR_i$. Therefore, it suffices to establish the following operator identities and a proof is given in the `Appendix'.
\begingroup
    \begin{align}\label{eqn:New01}
    & (a_1)\;\;  A_iA_{n-j}^*D_{P}+A_{n-i}^*D_{P}S_j =
A_jA_{n-i}^*D_{P}+A_{n-j}^*D_{P}S_i \notag \\
  & (a_2)\;\; S_iD_{P^*}{B_{n-j}}+D_{P^*}{B_{n-i}}{B_j}^*= S_jD_{P^*}{B_{n-i}}+D_{P^*}{B_{n-j}}{B_i}^* \notag \\
    & (a_3)\;\; A_iA_{n-j}^*P^* - A_{n-i}^*D_{P}D_{P^*}{B_{n-j}}+ A_{n-i}^*P^*{B_j}^*
= A_jA_{n-i}^*P^* - A_{n-j}^*D_{P}D_{P^*}{B_{n-i}} + A_{n-j}^*P^*{B_i}^*.
    \end{align}
    \endgroup

\noindent \textit{Step 2.} We now show that $R_iU=UR_i$ for each
$i=1,\dots,n-1$.
\begin{equation*}\label{2.5}
R_iU= \footnotesize \left[
\begin{array}{ c c c c|c|c c c c }
\bm{\ddots} &\vdots &\vdots&\vdots   &\vdots  &\vdots& \vdots& \vdots& \vdots\\
\cdots &0&A_i&A_{n-i}^* &0&0&0&0&\cdots\\
\cdots &0&0&A_i&A_{n-i}^*D_{P}&-A_{n-i}^*P^*&0&0&\cdots\\
\cdots&0&0&0&A_iD_{P}+A_{n-i}^*D_{P}P
&A_{n-i}^*D_{P}D_{P^*}-A_iP^*
& -A_{n-i}^*P^*&0&\cdots\\
\hline

\cdots&0&0&0 &S_iP& S_iD_{P^*}&D_{P^*}{B_{n-i}}&0&\cdots\\
\hline

\cdots&0&0 &0&0&0&  {B_i}^*& {B_{n-i}}&\cdots\\
\cdots&0&0 &0&0&0&0& {B_i}^*&\cdots\\

\cdots &0&0   &0&  0&0&0&0&\cdots\\
\vdots &\vdots&\vdots&\vdots&\vdots&\vdots&\bm{\ddots}\\
\end{array} \right]
\end{equation*}
and
\begin{equation*}\label{2.5}
UR_i= \footnotesize \left[
\begin{array}{ c c c c|c|c c c c }
\bm{\ddots} &\vdots &\vdots&\vdots   &\vdots  &\vdots& \vdots& \vdots& \vdots\\
\cdots &0&A_i&A_{n-i}^*&0 &0&0&0&\cdots\\
\cdots &0&0&A_i&A_{n-i}^*D_{P}&-A_{n-i}^*P^*&0&0&\cdots\\
\cdots &0&0&0&D_{P}S_i&D_{P}D_{P^*}{B_{n-i}}-P^*{B_i}^*
& -P^*{B_{n-i}}&0&\cdots\\
\hline

\cdots &0&0&0 &PS_i& PD_{P^*}{B_{n-i}}+D_{P^*}{B_i}^*&D_{P^*}{B_{n-i}}&0&\cdots\\
\hline

\cdots &0&0 &0&0&0&  {B_i}^*& B_{n-i}&\cdots\\
\cdots &0&0 &0&0&0&0& {B_i}^*&\cdots\\

\cdots &0&0   &0&  0&0&0&0&\cdots\\
\vdots &\vdots&\vdots&\vdots&\vdots&\vdots&\bm{\ddots}\\
\end{array} \right].
\end{equation*}
The equality of the entities in the positions $(-1,2),\,(-1,0)$
and $(0,1)$ of $R_iU$ and $UR_i$ follows from part-(1), part-(2)
and part-(3) of Lemma \ref{funda-properties}. Therefore, for
showing the equality of $R_1U$ and $UR_1$ we have to verify that
$A_{n-i}^*D_{P}D_{P^*}-A_iP^*=D_{P}D_{P^*}B_{n-i}-P^*{B_i}^*$. Let
$T=(A_{n-i}^*D_{P}D_{P^*}-A_iP^*)-(D_{P}D_{P^*}B_{n-i}-P^*{B_i}^*)$.
Then $T$ maps $\mathcal D_{P^*}$ into $\mathcal D_{P}$. Now
\begin{align*}
D_{P}TD_{P^*}
&=D_{P}A_{n-i}^*D_{P}D_{P^*}^2-D_{P}A_iP^*D_{P^*}+D_{P}P^*{B_i}^*D_{P^*}
\\& \quad -D_{P}^2D_{P^*}B_{n-i}D_{P^*}
\\&
=(S_{n-i}^*-P^*S_i)(I-PP^*)-(S_i-S_{n-i}^*P)P^*\\& \quad +P^*(S_i-PS_{n-i}^*)
-(I-P^*P)(S_{n-i}^*-S_iP^*) \\& =0.
\end{align*}
We used (\ref{nagy-foias}), (\ref{funda-repeat}) and
(\ref{funda-repeat}). Hence $R_iU=UR_i$.\\

\noindent \textit{Step 3.} We now show that
\[
\| f(R_1,\dots , R_{n-1},U) \| \leq \| f \|_{\infty , \Gamma_n}\;,
\text{ for all } f \in \mathbb C[z_1,\dots, z_n].
\]
We first show that each $R_i$ is normal. It suffices if we show
that $R_i=R_{n-i}^*U$. This is because if $R_i=R_{n-i}^*U$, then
$R_{n-i}=R_i^*U$. Now,
$
R_iR_{n-i}=R_i R_i^*U $ and $R_{n-i}R_i =
R_i^*UR_i=R_i^*R_iU$. Since $R_iR_{n-i}=R_{n-i}R_i$, we have that $R_iR_i^*=R_i^*R_i$.
\begin{eqnarray}
R_{n-i}^*U & = \notag 
 \SMALL \left[
\begin{array}{ c c c c|c|c c c c}
\bm{\ddots}&\vdots &\vdots&\vdots   &\vdots  &\vdots& \vdots&\vdots&\vdots\\
\cdots&A_i&A_{n-i}^*&0&0&0&0&0&\cdots\\
\cdots&0&A_i&A_{n-i}^*&0&0&0&0&\cdots\\
\hline

\cdots&0&0&D_{P}A_i&S_{n-i}^*&0&0&0&\cdots\\
\hline

\cdots&0&0&-PA_i&  {B_i}^*D_{P^*}& B_{n-i}&0&0&\cdots\\
\cdots&0&0&0&0&{B_i}^*&B_{n-i}&0&\cdots\\
\vdots&\vdots&\vdots&\vdots&\vdots&\vdots&\vdots&\vdots&\bm{\ddots}\\
\end{array} \right]
\left[
\begin{array}{ c c c|c|c c c }
\bm{\ddots} &\vdots&\vdots   &\vdots  &\vdots& \vdots&\vdots\\
\cdots&0&I  &0&  0&0&\cdots\\
\cdots&0&0  &D_{P}&  -{P}^*&0&\cdots\\ \hline

\cdots&0&0   &{P}&   D_{{P}^*}&0&\cdots\\ \hline

\cdots&0&0   &0&  0& I&\cdots\\
\cdots&0&0   &0&  0&0&\cdots\\
\vdots&\vdots&\vdots&\vdots&\vdots&\vdots&\bm{\ddots}\\
\end{array} \right] \notag \\&
= \SMALL \left[
\begin{array}{ c c c c|c|c c c c }
\bm{\ddots} &\vdots &\vdots&\vdots   &\vdots  &\vdots& \vdots& \vdots& \vdots\\
\cdots &0&A_i&A_{n-i}^* &0&0&0&0&\cdots\\
\cdots &0&0&A_i&A_{n-i}^*D_{P}&-A_{n-i}^*P^*&0&0&\cdots\\

\hline

\cdots&0&0&0 &S_{n-i}^*P+D_{P}A_iD_{P}& S_{n-i}^*D_{P^*}-D_{P}A_iP^*&0&0&\cdots\\
\hline

\cdots&0&0 &0&{B_i}^*D_{P^*}P-PA_iD_{P}& {B_i}^*D_{P^*}^2+PA_iP^*& B_{n-i}&0&\cdots\\
\cdots&0&0 &0&0&0&{B_i}^*& B_{n-i}&\cdots\\

\vdots &\vdots&\vdots&\vdots&\vdots&\vdots&\bm{\ddots}\\
\end{array} \right]
\end{eqnarray}

\noindent In order to prove $R_i=R_{n-i}^*U$, we need to show the
following steps because the other equalities follow from
(\ref{funda-repeat}) and (\ref{funda-repeat1}).
\begin{itemize}
\item[($c_1$)]${B_i}^*D_{P^*}P=PA_iD_{P^*}$, \item[($c_2$)]
$D_{P^*}B_{n-i}=S_{n-i}^*D_{P^*}-D_{P}A_iP^*$, \item[($c_3$)]
${B_i}^*D_{P^*}^2+PA_iP^*={B_i}^*$.
\end{itemize}

\noindent $(c_1)$. This follows from part-(1) of Lemma
\ref{funda-properties} together with (\ref{nagy-foias}).\\

\noindent $(c_2).$  Let $J_1=D_{P^*}B_{n-i}+D_{P}A_iP^*$. Now
\begin{align*}
J_1D_{P^*} & =D_{P^*}B_{n-i}D_{P^*}+D_{P}A_iP^*D_{P^*}
\\&=(S_{n-i}^*-S_iP^*)+D_{P}B_iD_{P}P^*\\&
=(S_{n-i}^*-S_iP^*)+(S_i-S_{n-i}^*P)P^*\\&=S_{n-i}^*D_{P^*}^2.
\end{align*}
We used (\ref{nagy-foias}), (\ref{funda-repeat}) and
(\ref{funda-repeat1}) here. Since $J$ is defined from $D_{P^*}$ to
$\mathcal H$, $(c_2)$ is
established.\\

\noindent $(c_3)$. ${B_i}^*D_{P^*}^2+PA_iP^*=
{B_i}^*(I-PP^*)+{B_i}^*PP^*={B_i}^*$.\\

Let $f(z_1,\dots,z_n)$ be a holomorphic polynomial. It is evident
from the block matrices of the operators $R_1.\dots,R_{n-1}, U$
that
\[
 f(R_1,\dots,R_{n-1},U)  =\left[
\begin{array}{ccc}
f(G_{11},\dots, G_{n-1\, 1}, Q_1) & * & *\\
0 & f(S_1,\dots,S_{n-1},P) & *\\
0& 0& f(G_{15},\dots, G_{n-1\,5},Q_5)  \end{array} \right].
\]
Since $R_1,\dots, R_{n-1},U$ are commuting normal operators, by
Fuglede's Theorem (see \cite{Fuglede}), they doubly commute and
thus $f(R_1,\dots , R_{n-1},U)$ is a normal operator. Therefore,
\[
\|f(R_1,\dots , R_{n-1},U)\|=r(f(R_1,\dots , R_{n-1},U)).
\]
Let us consider the following diagonal block matrices of operators defined on $\mathcal K$:
\[
\widetilde{R_i}=\left[
\begin{array}{ccc}
G_{i1} & 0 & 0\\
0 & S_{i} & 0\\
0& 0& G_{i5}  \end{array} \right]\;,\; \widetilde{U}=\left[
\begin{array}{ccc}
Q_1 & 0 & 0\\
0 & P & 0\\
0& 0& Q_5  \end{array} \right]\;, \quad 1\leq i \leq n-1.
\]
Since $\Sigma_1$, $\Sigma_2$ are $\Gamma_{n-1}$-contractions for every $z\in \T$, it follows from Theorem \ref{model1} that the tuples $(G_{11},\dots,G_{n-1\,1},Q_1)$ and
$(G_{15},\dots,G_{n-1\,5},Q_5)$ are $\Gamma_n$-isometry and $\Gamma_n$-co-isometry respectively. Therefore,
$(\widetilde{R_1},\dots,\widetilde{R_{n-1}},\widetilde{U})$ is a
$\Gamma_n$-contraction by being direct sum of three
$\Gamma_n$-contractions. Now
\[
 f(\widetilde{R_1}, \dots, \widetilde R_{n-1}, \widetilde{U})
=\left[
\begin{array}{ccc}
f(G_{11},\dots, G_{n-1\, 1}, Q_1) & 0 & 0\\
0 & f(S_1,\dots,S_{n-1},P) & 0\\
0& 0& f(G_{15},\dots, G_{n-1\,5},Q_5)  \end{array} \right].
\]
We now apply to $f(R_1,\dots,R_{n-1}, U )$, Lemma 1 in
\cite{hong}, which states that the spectrum of an operator of the
form $\begin{bmatrix} X&Y\\0&Z
\end{bmatrix}$ is a subset of $\sigma(X)\cup \sigma (Z)$.
So, we obtain
\begin{align*}
\sigma(f(R_1,\dots,R_{n-1},U)) & \subseteq
\sigma(f(G_{11},\dots, G_{n-1\, 1}, Q_1))\cup
\sigma(f(S_1,\dots,S_{n-1},P)) \\& \quad \cup
\sigma(f(G_{15},\dots, G_{n-1\,
5}, Q_5)) \\
& =
\sigma(f(\widetilde{R_1},\dots,\widetilde{R_{n-1}},\widetilde{U})).
\end{align*}
Now by the spectral mapping theorem in several variables, we have
that
\[
\sigma(f(\widetilde{R_1},\dots,\widetilde{R_{n-1}},\widetilde{U}))
=\{ f(\lambda_1,\dots,\lambda_n)\,:\,
(\lambda_1,\dots,\lambda_n)\in \sigma_T(\widetilde{R_1},\dots,
\widetilde{R_{n-1}}, \widetilde{U}) \}.
\]
Since $(\widetilde{R_1},\dots,\widetilde{R_{n-1}},\widetilde{U})$
is a $\Gamma_n$-contraction,
$\sigma_T(\widetilde{R_1},\dots,\widetilde{R_{n-1}},\widetilde{U})
\subseteq \Gamma_n$. Therefore,
\[
r(f(\widetilde{R_1},\dots,\widetilde{R_{n-1}},\widetilde{U}))\leq
\sup_{(z_1\dots,z_{n-1})\in\Gamma_n}\,|f(z_1,\dots,z_{n-1})|=\|f\|_{\infty,\Gamma_n}.
\]
Since
$
\sigma(f(R_1,\dots,R_{n-1},U)) \subseteq
\sigma(f(\widetilde{R_1},\dots,\widetilde{R_{n-1}},\widetilde{U}))\,,
$
we have that
$r(f(\widetilde{R_1},\dots,\widetilde{R_{n-1}},\widetilde{U}))$ is less than or equal to
$\|f\|_{\infty,\Gamma_n}$. It follows from here that
\[
\|f(\widetilde{R_1},\dots,\widetilde{R_{n-1}},\widetilde{U})
\|=r(f(\widetilde{R_1},\dots,\widetilde{R_{n-1}},\widetilde{U}))\leq
\|f\|_{\infty,\Gamma_n},
\]
and the proof of the reverse direction is complete.\\

\noindent \textbf{\textit{The} $( \Rightarrow )$ \textit{part}.}\\

\noindent To establish the forward direction of the theorem, we first ensure the uniqueness (upto a unitary) of the minimal unitary dilation as far as the dilation space is the minimal unitary dilation space of $P$. To do that it is enough to prove the following.
\begin{enumerate}
\item If $(\widehat{R}_1,\dots,\widehat{R}_{n-1},U)$ on $\mathcal{K}$ is a $\Gamma$-unitary dilation of $(S_1,\dots,S_{n-1},P)$, then
$\widehat{R}_i=R_i$ for $i=1,\dots ,n-1$.
\item If $(\overline{R}_1,\dots, \overline{R}_{n-1},\overline{U})$ is a $\Gamma_n$-unitary dilation of $(S_1,\dots,S_{n-1},P)$, where
$\overline{U}$ is a minimal unitary dilation of $P$, then $(\overline{R}_1,\dots, \overline{R}_{n-1},\overline{U})$ is unitarily equivalent to $(R_1,\dots, R_{n-1},U)$.
\end{enumerate}
Needless to mention that without loss of generality we can choose $\mathcal K$ to be the Schaeffer minimal unitary dilation space $ l^2(\mathcal{D}_P) \oplus \mathcal{H} \oplus l^2(\mathcal{D}_{P^*})$ and $U$ on $\mathcal K$ to be the Schaeffer minimal unitary dilation of $P$ as described in the previous part of the proof. Suppose $(\widehat{R}_1,\dots,\widehat{R}_{n-1},U)$ on $\mathcal{K}$ is a $\Gamma$-unitary dilation of $(S_1,\dots,S_{n-1},P)$. We first show that each $\widehat R_i$ admits an upper triangular block matrix of the following form:
\[
\widehat R_i=
\begin{pmatrix}
* & * & * \\
0 & S_i & * \\
0 & 0 & * \\
\end{pmatrix}, \quad \text{with respect to } \;\; \mathcal{K}=l^2(\mathcal{D}_P) \oplus \mathcal{H} \oplus l^2(\mathcal{D}_{P^*}).
\]
This is indeed a general fact: if $(\widehat{R}_1, \dots, \widehat R_{n-1},U)$ is a dilation (not necessarily a $\Gamma_n$-unitary one) of $(S_1,\dots, S_{n-1},P)$ whence $U$ is the Schaeffer minimal unitary dilation of $P$, the block matrices of each $\widehat R_i$ is of such upper triangular form. A proof to this general fact could be found in Chapter-II of \cite{nagy}. Nevertheless, we include a proof here for the convenience of a reader. Let us denote $\widehat{\mathcal{H}}=l^2(\mathcal{D}_P) \oplus \mathcal{H}$. Since $U$ is the minimal unitary dilation of $P$, we have that
\[
\mathcal{K}=\bigvee_{m=-\infty}^{\infty}U^m\mathcal{H} \quad \text{ and } \quad \widehat{\mathcal{H}}=\bigvee_{m=0}^{\infty}U^m\mathcal{H}=\bigvee_{m=0}^{\infty}V^m\mathcal{H},
\]
where $V$ is the Schaeffer minimal isometry dilation of $P$ on $\widehat{\HS}$.
Again for $i=1,\dots ,n-1$ we have
\[
P_{\mathcal{H}} \widehat{R}_i(U^mh)=S_iP^mh= S_iP_{\mathcal{H}}U^mh, \text{ for all } h \in \mathcal{H}, m \in \mathbb{N} \cup \{ 0 \}.
\]
Set $\widehat{R}_i=(R^i_{tk})_{t,k=1}^3$ with respect to $\mathcal{K}=l^2(\mathcal{D}_P) \oplus \mathcal{H} \oplus l^2(\mathcal{D}_{P^*})$ for each $i$ and consider the block matrix of $U$ from (\ref{sec:eqn-011}). Then we have $P_{\mathcal{H}}\,\widehat R_i|_{\widehat{\mathcal{H}}}=S_iP_{\mathcal{H}}|_{\widehat{\mathcal{H}}}$ or equivalently $S_i^*=P_{\widehat{\mathcal{H}}}\, \widehat R_i^*|_{\mathcal{H}}$ for $1 \leq i \leq n-1$. This shows that $R^i_{21}=0$. Set $\widehat{\mathcal{H}}_1=\mathcal{H} \oplus l^2(\mathcal{D}_{P^*})$, then note that $\widehat{\mathcal{H}}_1=\bigvee_{n=0}^{\infty}{U^*}^n\mathcal{H}$.
So, for $i=1,\dots, n-1$ we have that
\[
P_{\mathcal{H}} \widehat R^*_i({U^*}^mh)=S^*_i{P^*}^mh= S^*_iP_{\mathcal{H}}{U^*}^mh, \text{ for all } h \in \mathcal{H}, m \in \mathbb{N} .
\]
Thus, $S_i=P_{\widehat{\mathcal{H}}_1}\, \widehat{R}_i|_{\mathcal{H}}$. Therefore, $R^i_{32}=0$, for each $i$. Now we show that $R^j_{13}=0$. Since $\widehat R_i$ commutes with $U$, considering the block matrix form of $U$ as in (\ref{sec:eqn-011}) an easy calculation gives
\begin{eqnarray}\label{sec:eqn-002}
R^i_{31}Q_1=Q_5R^i_{31} \text{ and } R^i_{31}Q_2=0,
\end{eqnarray}
This follows by comparing the $(3,1)$ and $(3,2)$ entries of $\widehat R_iU$ and $U \widehat R_i$ respectively.
Again $Ran \, Q_2=Ran(I-Q_1Q_1^*)$. Therefore, $R^i_{31}(I-Q_1Q_1^*)=0$. Thus from (\ref{sec:eqn-002}) it follows that
$R^i_{31} = Q_5R^i_{31}Q_1^*$. This gives after finitely many iteration $R^i_{31}=Q_5^nR^i_{31}{Q_1^*}^n$. Now since ${Q_1^*}^n \to 0$ as $n \to \infty$, we have that $R^i_{31}=0$ for all $i=1,2,\dots, n-1$ and thus we are done.\\

As a next step we prove the following claim.\\

\noindent \textit{Claim 1.}
If $({R}_1^{'},\dots, {R}_{n-1}^{'},U)$ on $\mathcal K$ is another $\Gamma_n$-unitary dilation of $(S_1,\dots, S_{n-1},P)$ such that ${R}_i^{'}|_{\widehat{\mathcal H}}=R_i|_{\widehat{\mathcal H}}$, $\widehat{\mathcal H}=l^2(\mathcal D_P)\oplus \mathcal H$, then $R_i^{'}=R_i$ for $i=1,\dots, n-1$.\\

\noindent \textit{Proof of Claim 1.} Let $V_i^{'}={R}_i^{'}|_{\widehat{\mathcal H}}=R_i|_{\widehat{\mathcal H}}$ for $i=1,\dots, n-1$. Then the block matrices of $R_i^{'}$ and $U$ with respect to the orthogonal decomposition $\mathcal K=\widehat{\mathcal H}\oplus \mathcal{\mathcal D_{P^*}}$ are upper triangular, say
\[
R_i^{'}=\begin{bmatrix}
V_i^{'} & \tilde{C}_i \\
0 & \tilde{D}_i \\
\end{bmatrix}
\quad \text{ and } \quad U =\begin{bmatrix}
V & {C} \\
0 & {D} \\
\end{bmatrix},
\]
where $V=U|_{\widehat{\mathcal H}}$ and $i=1,\dots , n-1$. If for each $i$ the block matrices of $R_i$ with respect to the same decomposition is
\[
R_i=\begin{bmatrix}
V_i^{'} & {C}_i \\
0 & {D}_i \\
\end{bmatrix},
\]
Since $({R}_1^{'},\dots,{R}_2^{'},U)$ is a $\Gamma_n$-unitary, we have by Theorem \ref{G-unitary} that
\[
R_i^{'}R_j^{'}=R_j^{'}R_i^{'},  \quad  R_i^{'}U=UR_i^{'} \quad \& \quad R_i^{'}=R_{n-i}^{'*}U \quad \forall i,j.
\]
Since $U$ is unitary, we have
\begin{eqnarray}\label{gamma1}
D^*D+C^*C=I \text{ and }C^*V=0.
\end{eqnarray}
The fact that $R_i^{'}$ commutes with $U$ gives us
\begin{eqnarray}\label{gamma23}
V_i^{'}C+\tilde{C}_iD=V\tilde{C}_i+C\tilde{D}_i \text{ and } \tilde{D}_iD=D\tilde{D}_i.
\end{eqnarray}
Also the commutativity of $R_i^{'}$ and $R_j^{'}$ gives
\begin{eqnarray}\label{gamma12}
V_i^{'}\tilde{C}_j+\tilde{C}_i\tilde{D}_j=V_j^{'}\tilde{C}_i+\tilde{C}_j\tilde{D}_i \text{ and } \tilde{D}_i\tilde{D}_j=\tilde{D}_j\tilde{D}_i.
\end{eqnarray}
Again $R_i^{'}=R_{n-i}^*U$ implies that
\begin{eqnarray}\label{gamma123}
\tilde{C}_i=V_{n-i}^{'*}C \text{ and } \tilde{D}_i=\tilde{C}_{n-i}^*C+\tilde{D}_{n-i}^*D.
\end{eqnarray}
Therefore,
\begin{eqnarray*}
\tilde{C}_i(d_0,d_1,\dots)
&=&
V_{n-i}^{'*}C(d_0,d_1,\dots)
\\
&=&
V_{n-i}^{'*}(D_{P^*}d_0 \oplus (-P^*d_0,0,0,\dots))
\\
&=&
((S_{n-i}^*D_{P^*}-D_PA_iP^*)d_0 \oplus (-A_{n-i}^*P^*d_0,0,0,\dots))
\\
&=&
(D_{P^*}B_{n-i}d_0 \oplus (-A_2^*P^*d_0,0,0,\dots)) \quad [\text{by part-(3) of Lemma \ref{funda-properties}}] \\
&=&
C_i(d_0,d_1,\dots).
\end{eqnarray*}
Thus, $\tilde{C}_i=C_i$ for $i=1,\dots, n-1$. We now show that $\tilde D_i=D_i$ for each $i$. Multiplying (\ref{gamma23}) from left by $C^*$ and using (\ref{gamma1}), we have
$
\tilde{D}_{i}^*(I-D^*D)=D^*\tilde{C}_{i}^*C+C^*V_{i}^{'*}C.
$
Since $(I-D^*D)$ is the orthogonal projection of $l^2(\mathcal{D}_{P^*})$ onto $\mathcal D_{P^*}\oplus \{0\}\oplus \{0\}\oplus \dots$, for $d\in\mathcal{D}_{P^*}$, we have
\begin{equation}\label{eqn:New02}
\tilde{D}_i^*(d,0,0,\dots)=(B_id,B_{n-i}^*d,0,0,\dots).
\end{equation}
A proof of (\ref{eqn:New02}) is given in the `Appendix'. Now
\begingroup
\begin{align*}
\tilde{D}_i^*(\underbrace{0,\dots,0}_\text{$n$ times},d,0,\dots )
& =  \tilde{D}_i^*{D^*}^n(d,0,0,\dots) \\
&= {D^*}^n\tilde{D}_i^*(d,0,0,\dots) [\text{ using (\ref{gamma23})}] \\
&= {D^*}^n(B_id,B_{n-i}^*d,0,0,\dots) \\
&= (\underbrace{0,\dots,0}_{\text{$n$ times}},B_id,B_{n-i}^*d,0,0,\dots), \text{ for every } n \geq 0.
\end{align*}
\endgroup
Again, for $(c_0,c_1,c_2,\dots) \in l^2(\mathcal{D}_{P^*})$, we have
\begin{eqnarray*}
\tilde{D}_i^*(c_0,c_1,c_2,\dots)&=&\tilde{D}_i^*((c_0,0,0,\dots) + (0,c_1,0,\dots) + (0,0,c_2,\dots) +\cdots )
\\
&=&
(B_ic_0,B_{n-i}^*c_0,0,\dots)  + \;  (0,B_ic_1,B_{n-i}^*c_1,0,0,\dots) \\
&& + (0,0,B_ic_2,B_{n-i}^*c_2,0,0,\dots) + \cdots
\\
&=&
(B_ic_0,B_{n-i}^*c_0+B_ic_1,B_{n-i}^*c_1+B_jc_2,\dots).
\end{eqnarray*}
Also, we have (see the `Appendix' for a proof)
\begin{align}\label{eqn:New03}
\tilde{D}_i(c_0,c_1,c_2,\dots) &= (B_i^*c_0+B_{n-i}c_1,B_i^*c_1+B_{n-i}c_2,B_i^*c_2+B_{n-i}c_3,\dots) \notag \\
& =D_i(c_0,c_1,c_2,\dots),
\end{align}
and consequently $\tilde{D}_i=D_i$ for each $i=1.\dots, n-1$. Hence $R_i^{'}=R_i$ for $i=1,\dots, n-1$ and proof of \textit{Claim 1} is complete.\\

\noindent \textit{Claim 2.} If $(\widehat V_1,\dots, \widehat V_{n-1},V)$ on $\widehat{\mathcal H}=\mathcal H \oplus l^2(\mathcal D_P)$ dilates $(S_1,\dots, S_{n-1},P)$ and satisfies $\widehat V_i= \widehat V_{n-i}^*V$ for each $i$, then $\widehat{V_i}=V_i$ for $i=1, \dots , n-1$. Here $V_i=R_i|_{\widehat{\mathcal H}}$ and $V=U|_{\widehat{\mathcal H}}$.\\

\noindent \textit{Proof of Claim 2.} Let
$V=\begin{pmatrix} P&0\\C_1&D_1
\end{pmatrix}$
with respect to the decomposition $\widehat H=\mathcal H\oplus l^2(\mathcal {D}_P)$, where
\begin{align*}
C_1 &=\begin{bmatrix} {D}_P
\\0\\0\\ \vdots \end{bmatrix}: \mathcal H \rightarrow
l^2(\mathcal{D}_P)
\quad \text{ and } \\ D_1 &=\begin{bmatrix}
0&0&0&\dots\\I&0&0&\dots\\0&I&0&\dots\\\dots&\dots&\dots&\dots
\end{bmatrix}: l^2(\mathcal{D}_P) \rightarrow l^2(\mathcal{D}_P).
\end{align*}
By an argument similar to that of the previous part, we have that each $\widehat{V_i}$ on $\mathcal
H\oplus l^2(\mathcal D_P)$ has block-matrix form $\widehat{V_i}=
\begin{bmatrix} S_i&0\\E_i&F_i
\end{bmatrix}$. Let us consider the natural unitary
\begin{align*}
U:\; & l^2(\mathcal D_P) \rightarrow H^2(\mathcal D_P)
\\& (\displaystyle
\underbrace{0,0,\dots,0}_n,1,0,0,\dots) \mapsto z^n.
\end{align*}
Then $\widehat{V_i}$ and $V$ on $\mathcal H \oplus l^2(\mathcal D_P)$ are respectively identified with the operators
\[
\widetilde{V_i}=\begin{bmatrix} S&0\\UE_i&UF_iU^* \end{bmatrix} \text{ and } \widetilde{V}=\begin{bmatrix} P&0\\UC_1&UD_1U^* \end{bmatrix}
\text{ on } \mathcal H\oplus H^2(\mathcal D_P).
\]
Evidently $UD_1U^*$ is the multiplication
operator $M_z^{\mathcal D_P}$ on $H^2(\mathcal D_P)$. Since $\widetilde V_1,\dots, \widetilde V_{n-1}, \widetilde V$ commute, so do $UF_iU^*$ and $UD_1U^*=M_z^{\mathcal D_P}$ for each $i$. By being a commutant of $M_z^{\mathcal D_P}$, each $UF_iU^*$ takes the form $M_{\varphi}^{\mathcal D_P}$ for some $\varphi_i\in
H^{\infty}(\mathcal B (\mathcal D_P))$. Again by the relations $\widehat V_i=\widehat V_{n-i}^*V$ we have that $\widetilde V_i=\widetilde V_{n-i}^* \widetilde V$ for each $i$. This implies that
\begin{align*}
\begin{bmatrix} S_i&0\\UE_i&M_{\varphi_i}^{\mathcal D_P} \end{bmatrix} & =
\begin{bmatrix} S_{n-i}^*&E_{n-i}^*U^*\\0&{M_{\varphi_{n-i}}^{\mathcal D_P}}^* \end{bmatrix}
\begin{bmatrix} P&0\\UC_1&M_z^{\mathcal D_P} \end{bmatrix} \\
& =
\begin{bmatrix} S_{n-i}^*P+E_{n-i}^*C_1&E_{n-i}^*U^*M_z^{\mathcal D_P}\\{M_{\varphi_{n-i}}^{\mathcal D_P}}^*UC_1&{M_{\varphi_{n-i}}^{\mathcal D_P}}^*M_z^{\mathcal D_P}
\end{bmatrix},
\end{align*}
which gives
\begin{eqnarray}\begin{cases} \label{e9}
&(\mbox{i})\; S_i-S_{n-i}^*P=E_{n-i}^*C_1\\&(\mbox{ii})\;
UE={M_{\varphi_{n-i}}^{\mathcal D_P}}^*UC_1\\&(\mbox{iii})\;
M_{\varphi}^{\mathcal D_P}={M_{\varphi_{n-i}}^{\mathcal D_P}}^*M_z.
\end{cases}
\end{eqnarray}
From (\ref{e9})-(iii), it is clear by considering the power
series expansion that $\varphi(z)=A_i^0+A_{n-i}^{0^*}z$, for some $A_1^0, \dots, A_{n-1}^0\in\mathcal B(\mathcal
 D_P)$. It is obvious that if $D_i^0=\begin{bmatrix} A_i^0&0&0&\dots\\ A_{n-i}^{0*}&A_i^0&0&\dots\\0&A_{n-i}^{0*}&A_i^0&\dots\\\dots&\dots&\dots&\dots \end{bmatrix}$
on $l^2(\mathcal D_P)$,
then $UD_i^0U^*=M_{\varphi_i}^{\mathcal D_P}$. Therefore, it follows that $F_i=D_i^0$ for $i=1,\dots ,n-1$.
Combining this with (\ref{e9})-(ii), we get
$UE_i={M_{\varphi_{n-i}}^{\mathcal D_P}}^*UC_1=UD_{n-i}^{0*}U^*UC_1=UD_{n-i}^{0*}C_1$,
i.e, $E=D_{n-i}^{0*}C_1$. Therefore,
\[
\widehat{V_i}=\begin{bmatrix}
S_i&0\\D_{n-i}^{0*}C_1&D_i^0
\end{bmatrix} \mbox{ on } \mathcal H\oplus l^2(\mathcal D_P)\,, \quad i=1,\dots, n-1.
\]
Considering the above stated matrix forms of $D_i^0$ and $C_1$ we get
\[
D_i^{0*}C_1=\begin{bmatrix} A_i^{0*} D_P
\\0\\0\\\vdots \end{bmatrix}.
\]
Hence with respect to the
decomposition $\widehat H=\mathcal H\oplus l^2(\mathcal D_P)$, we have for each $i$
\[
\widehat{V_i}=\begin{bmatrix} S_i&0&0&0&\dots\\ A_{n-i}^{0*} D_P &A_i^0&0&0&\dots\\
0& A_{n-i}^{0*}& A_i^0&0&\dots\\0&0&A_{n-i}^{0*}& A_i^0&\dots\\\dots&\dots&\dots&\dots&\dots
\end{bmatrix}.
\]
Hence $A_i=A_i^0$ for each $i$ and consequently $\widehat{V}_i=V_i$ for $i=1,\dots ,n-1$. The proof of \textit{Claim 2} is now complete.

It is evident now that the proof of \textit{Part-1} follows from \textit{Claim 1} and \textit{Claim 2}. Now we need to prove \textit{Part-2}. If $(\overline{R}_1,\dots, \overline{R}_{n-1},\overline{U})$ on $\underline{\mathcal K}$ is a $\Gamma_n$-unitary dilation of $(S_1,\dots,S_{n-1}P)$ where
$\overline{U}$ is a minimal isometric dilation of $P$, then there exists a unitary $U_1: \underline{\mathcal K} \to \mathcal{K}$ such that $U_1\overline{U}U_1^*=U$ and $U_1h=h$ for all $h \in \mathcal{H}$. This shows that $(U_1\overline{R_1}U_1^*,\dots,U_1\overline{R}_{(n-1)}U_1^*,U_1\overline{U}U_1^*)$ on $\mathcal K$ is another $\Gamma_n$-unitary dilation of $(S_1,\dots, S_{n-1},P)$, where $U_1\overline{U}U_1^*=U$. Hence by \textit{Part-1}, we have that
\[
(U_1\overline{R_1}U_1^*,\dots,U_1\overline{R}_{(n-1)}U_1^*,U_1\overline{U}U_1^*)=(R_1,\dots,R_{n-1},U).
\]
Hence $(\overline{R}_1,\dots, \overline{R}_{n-1},\overline{U})$ is unitarily equivalent to $(R_1,\dots, R_{n-1},U)$ and the proof of \textit{Part-2} is finished. The rest of the proof follows from the structures of the block matrices of $R_i$ and $U$ as in (\ref{2.3}) and (\ref{2.33}). The restriction of $(R_1, \dots , R_{n-1},U)$ to the joint-invariant subspace $l^2(\mathcal D_P)$ is a $\Gamma_n$-isometry, which is unitarily equivalent to the Toeplitz operator tuple $\left(T_{\phi_1}, \dots , T_{\phi_{n-1}},T_z    \right)$ on $H^2(\mathcal D_P)$, where $\phi_i(z)=A_i+A_{n-i}^*z \;$ for each $i$. Thus, it follows from Lemma \ref{lem:BSR2} that
\[
 \left(
\frac{n-1}{n}(A_1+A_{n-1}^*z),\frac{n-2}{n}(A_2+A_{n-2}^*z),\dots,
\dfrac{1}{n}(A_{n-1}+A_1^*z) \right)
\]
is a $\Gamma_{n-1}$-contraction for every $z\in \T$. Also, considering the restriction of $(R_1^*, \dots , R_{n-1}^*,U^*)$ to the joint-invariant subspace $l^2(\mathcal D_{P^*})$ and using similar argument as above, we have that
\[
\left(
\frac{n-1}{n}(B_{1}+B_{n-1}^*z),\frac{n-2}{n}(B_{2}+B_{n-2}^*z),\dots,
\dfrac{1}{n}(B_{n-1}+B_{1}^*z) \right)
\]
is a $\Gamma_{n-1}$-contraction for every $z\in \T$. The proof of the theorem is now complete.

\end{proof}

\begin{rem}
As we have mentioned before that Theorem \ref{main-dilation-theorem} generalizes the $\Gamma_2$-unitary dilation of a $\Gamma_2$-contraction obtained in \cite{sourav}. The literature tells us that every $\Gamma_2$-contraction $(S,P)$ dilates unconditionally to a $\Gamma_2$-unitary (see \cite{ay-jfa, tirtha-sourav, sourav}). The conditions that are there in the hypotheses of Theorem \ref{main-dilation-theorem} are automatically satisfied by the fundamental operator of a $\Gamma_2$-contraction.
\end{rem}

\subsection{A necessary and sufficient condition for dilation to a special class} In this subsection, we deal with dilation of $\Gamma_n$-contractions which are of the form $(S_1,\dots, S_{n-1},0)$. Since the last component $P=0$ here, we have $\mathcal D_P=\mathcal H$ and thus the minimal dilation space as in Theorem \ref{main-dilation-theorem} becomes $\mathcal K=l^2(\mathcal H)$.

\begin{thm}\label{dilation:char}
A $\Gamma_n$-contraction $\Upsilon = (S_1,\dots,S_{n-1},0)$ on a
Hilbert space $\mathcal H$ dilates to a $\Gamma_n$-unitary $(R_1, \dots , R_{n-1},U)$ on $l^2(\mathcal H)$ with $U$ being the minimal unitary dilation of $0$ if and only if
\[
\left( \dfrac{n-1}{n}(S_1+S_{n-1}^*z),
\dfrac{n-2}{n}(S_2+S_{n-2}^*z),\dots, \dfrac{1}{n}(S_{n-1}+S_1^*z)
\right)
\]
is a $\Gamma_{n-1}$-contraction for every $z\in \mathbb
T$.
\end{thm}

\begin{proof}
It is evident that $(S_1,\dots,S_{n-1})$ and $(S_1^*,\dots ,S_{n-1}^*)$ are the $\ft$-tuples of $\Upsilon$ and $\Upsilon^*$ respectively. So if $$\left( \dfrac{n-1}{n}(S_1+S_{n-1}^*z),
\dfrac{n-2}{n}(S_2+S_{n-2}^*z),\dots, \dfrac{1}{n}(S_{n-1}+S_1^*z)
\right)$$ is a $\Gamma_{n-1}$-contraction for every $z\in \mathbb
T$, then
\[
\left( \dfrac{n-1}{n}(S_1^*+S_{n-1} z),
\dfrac{n-2}{n}(S_2^*+S_{n-2} z),\dots, \dfrac{1}{n}(S_{n-1}^*+S_1 z)
\right)
\]
is also a $\Gamma_{n-1}$-contraction for all $z\in \mathbb T$. Therefore, by Theorem \ref{main-dilation-theorem}
$\Sigma$ dilates to a $\Gamma_n$-unitary on $l^2(\mathcal H)$.\\

Conversely, suppose $\Upsilon$ dilates to a $\Gamma_n$-unitary on $l^2(\mathcal H)$. Since $l^2(\mathcal H)$ is the minimal dilation space for $0$
Then by the last part of Theorem \ref{main-dilation-theorem}, $(R_1,\dots, R_{n-1},U)$ (as in Theorem \ref{main-dilation-theorem}) is a $\Gamma_n$-unitary dilation for $\Sigma$. By Lemma \ref{lem:BSR2}, $\left(\dfrac{n-1}{n}R_1, \dfrac{n-2}{n}R_2,\dots , \dfrac{1}{n} R_{n-1} \right)$ is a $\Gamma_{n-1}$-contraction. It is evident from the block matrix of $\left(\dfrac{n-1}{n}R_1, \dfrac{n-2}{n}R_2,\dots , \dfrac{1}{n} R_{n-1} \right)$ that
\[
\left( \dfrac{n-1}{n}(S_1+S_{n-1}^*z),
\dfrac{n-2}{n}(S_2+S_{n-2}^*z),\dots, \dfrac{1}{n}(S_{n-1}+S_1^*z)
\right)
\]
is a $\Gamma_{n-1}$-contraction for every $z\in \mathbb
T$. The proof is now complete.

\end{proof}

\subsection{Examples of $\Gamma_n$-contractions which dilate without any hypothesis}\label{further-dilation}

In this Subsection, we construct two different classes of $\G$-contractions which dilate to $\G$-unitaries without any hypothesis. This proves that the conditions of Theorem \ref{main-dilation-theorem} are sufficient but not necessary for the existence of a normal $b\G-$dilation of a $\Gamma_n$-contraction.

\begin{eg}\label{eg:N01}

The example that we are going to give here requires several steps. Let us consider the map
\begin{align*}
\varrho \,:\,& \mathbb C^2 \rightarrow \mathbb C^3 \\&
(z_1,z_2)\mapsto (z_1,z_2,0),
\end{align*}
that embeds $\Gamma_2$ analytically inside $\Gamma_3$ in the following way.
\begin{lem}
Let $\Gamma_3^0=\{(s_1,s_2,p)\in\Gamma_3: p=0 \}$. Then
$\varrho(\Gamma_2)=\Gamma_3^0$.
\end{lem}
\begin{proof}
We have that $\varrho(z_1,z_2)=(z_1,z_2,0)$ for all $(z_1,z_2)$ in
$\mathbb C^2$. Let $(s,p)\in \Gamma_2$. Then there are points
$\lambda_1,\lambda_2$ in the closed unit disc $\overline{\mathbb
D}$ such that $(s,p)=(\lambda_1+\lambda_2,\lambda_1\lambda_2)$.
Now clearly the point $(s,p,0)$, which is the image of $(s,p)$
under $\varrho$, is the symmetrization of the points
$\lambda_1,\lambda_2,0$ of $\overline{\mathbb D}$. Therefore,
$(s,p,0)\in\Gamma_3$ and in particular $(s,p,0)$ is in
$\Gamma_3^0$.

Conversely, let $(s_1,s_2,0)\in\Gamma_3^0$. Since $(s_1,s_2,0)$ is
a point of $\Gamma_3$, there are points $z_1,z_2,z_3$ in
$\overline{\mathbb D}$ such that $
\pi_3(z_1,z_2,z_3)=(s_1,s_2,0)$. Now $z_1z_2z_3=0$ implies that at
least one of $z_1,z_2,z_3$ is $0$. Let us assume without loss of
generality that $z_3=0$. Then $s_1=z_1+z_2$ and $s_2=z_1z_2$. This
shows that $(s_1,s_2)\in\Gamma_2$. Hence the proof is complete.

\end{proof}

\begin{lem}\label{implem}
If $(S,P)$ is a $\Gamma_2$-contraction then $(S,P,0)$ is a $\Gamma_3$-contraction.
\end{lem}

\begin{proof}
Let $f$ be a polynomial in $z_1,z_2,z_3$ and let
$f_1(z_1,z_2)=f(z_1,z_2,0)$. Then $f_1$ is a polynomial in $z_1,z_2$ and $f_1(z_1,z_2)=f\circ \varrho(z_1,z_2)$
Now
\begin{align*} \|f(S,P,0)\|=\| f_1(S,P)\| & \leq \|f_1\|_{\infty,
\Gamma}, \quad \textup{since }(S,P) \textup{ is a }\Gamma_2\textup{-contraction},\\
& =\|f\circ \varrho\|_{\infty,\Gamma_2} =\|f\|_{\infty,
\varrho(\Gamma_2)} \leq \|f\|_{\infty, \Gamma_3}.
\end{align*}
Therefore, $(S,P,0)$ is a $\Gamma_3$-contraction.
\end{proof}

The following result is due to Agler and Young. We shall apply this result in this section.

\begin{thm}[\cite{ay-jot}, Corollary 1.9]\label{nicethm}
Let $(S,P)$ be a pair of commuting operators such that $\|P\|<1$
and the spectral radius of $S$ is less than $2$. Then $(S,P)$ is a
$\Gamma$-contraction if and only if
$\omega(D_P^{-1}(S-S^*P)D_P^{-1})\leq 1. $
\end{thm}

\begin{thm}\label{thm:nice1}
If $(S,P)$ is a $\Gamma$-contraction then $(S,P,0)$ is a
$\Gamma_3$-contraction and it dilates to a $\Gamma_3$-unitary.
\end{thm}

\begin{proof}
 It suffices if we show that $\Gamma_3$ is a complete spectral set for $(S,P,0)$. Since rational dilation succeeds on $\Gamma_2$, it follows that $\Gamma$ is a spectral set
for $(S,P)$ if and only if $\Gamma$ is a complete spectral set for $(S,P)$. Let $\boldsymbol f =[f_{ij}]_{m\times n}$ be a
matricial polynomial in three variables $z_1,z_2,z_3$. Let
$\widehat{f}_{ij}(z_1,z_2)=f_{ij}(z_1,z_2,0)$ and
$\boldsymbol{\widehat f}=[\widehat f_{ij}]_{m\times n}$. Now
\begin{align*}
\|\boldsymbol f(S,P,0) \|=\|[f_{ij}(S,P,0)]_{m\times n} \| = \|[\widehat f_{ij} (S,P)]_{m\times n} \| & \leq
\sup_{(z_1,z_2)\in\Gamma}\|\boldsymbol{\widehat f}(z_1,z_2)\| \\&
=\sup_{(z_1,z_2)\in\Gamma}\|[\widehat f_{ij}(z_1,z_2)] \| \\&
=\sup_{(z_1,z_2,0)\in\Gamma_3}\|[\widehat f_{ij}(z_1,z_2,z_3)]\|
\\& \leq \|\boldsymbol f \|_{\infty, \Gamma_3}\,.
\end{align*}
Thus $\Gamma_3$ is a complete spectral set for $(S,P,0)$ and the proof is complete.
\end{proof}

So we have a class of $\Gamma_3$-contractions which always have
$\Gamma_3$-unitary dilation. One can easily verify that the $\ft$-tuple of the $\Gamma_3$-contraction $(S,P,0)$ is $(S,P)$. Now if we
choose $P$ to be non-normal with $\|P\|<1$ and $S$ to be normal
with norm of $S$ being sufficiently small so that the norm of
$D_P^{-1}(S-S^*P)D_P^{-1}$ is less than $1$. Then by Theorem
\ref{nicethm}, $(S,P)$ is a $\Gamma_2$-contraction. Since $S$ is
normal and $P$ is non-normal we have
\[
S^*S-SS^*\neq P^*P-PP^*
\]
and hence the $\ft$-tuple $(S,P)$ does not satisfy the condition that $\left( \dfrac{2}{3}(S+P^*z), \dfrac{1}{3}(P+S^*z) \right)$ is a $\Gamma_2$-contraction because otherwise by commutativity of the pair $\left( \dfrac{2}{3}(S+P^*z), \dfrac{1}{3}(P+S^*z) \right)$, $S,P$ will satisfy $S^*S-SS^*=P^*P-PP^*$. So we get a class of $\Gamma_3$-contractions that dilate to the distinguished boundary $b\Gamma_3$ despite the fact that their $\ft$-tuples do not satisfy the conditions of Theorem \ref{main-dilation-theorem}.

\end{eg}

\begin{eg}
In \cite{JH}, Holbrook has shown that the multivariate von
Neumann inequality holds for any number of $2\times 2$ commuting
contractive matrices, i.e, if $C_1,\hdots,C_n$ are commuting
$2\times 2$ matrices with $\|C_k\|\leq 1$ for all $k$ and if
$f:\overline{\mathbb D^n} \rightarrow \overline{\mathbb D}$ is analytic, then
$\|f(C_1,\hdots,C_n) \|\leq 1$. Moreover, any $n$-tuple of
commuting $2\times 2$ contractions has simultaneous commuting
unitary dilation. See Proposition 2 and Proposition 3 in \cite{JH}
for a proof to these results.

So, $\overline{\mathbb D^n}$ is a spectral set for any $2\times 2$ commuting contractions $C_1,\dots,C_n$. Let
$(U_1,\dots, U_n)$ be a commuting unitary dilation of
$(C_1,\dots, C_n)$. It is obvious that the symmetrization of
$U_1,\dots, U_n$, i.e, $\pi_n(U_1,\dots, U_n)$ is a $\Gamma_n$-unitary dilation of the $\Gamma_n$-contraction $\pi_n(C_1,\dots, C_n)$. Also, not every $\G$-contraction arises as symmetrization of commuting contractions (see \cite{spal3}). Whether every $\G$-contraction consisting of commuting matrices dilates to a $\G$-unitary or not remains an open problem.
\end{eg}

\vspace{0.4cm}

\section{An explicit $\Gamma_n$-isometric dilation and a concrete functional model for $\Gamma_n$-contractions} \label{sec:isometric-dilation-1}

\vspace{0.4cm}

\noindent In this Section, we produce a $\G$-isometric dilation of a $\G$-contraction by an application of the $\G$-unitary dilation that we have constructed in Section \ref{conditional-dilation}. The $\G$-isometric dilation will act on the minimal isometric dilation space $\HS \oplus l^2(\mathcal D_P)$ of $P$. Indeed, we shall prove that $H\oplus l^2(\mathcal D_P)$ is a joint-invariant subspace for the $\G$-unitary $(R_1, \dots , R_{n-1},U)$ on $l^2(\mathcal D_P) \oplus \HS \oplus l^2(\mathcal D_{P^*})$ from Section \ref{conditional-dilation}. As a consequence of the $\G$-isometric dilation, we obtain a functional model for such $\G$-contractions. We begin with a pair of preparatory results.

\begin{prop}\label{exist-minimal}
If a $\Gamma_n$-contraction $(S_1,\dots,S_{n-1},P)$ defined on $\mathcal H$
has a $\Gamma_n$-isometric dilation, then it has a minimal
$\Gamma_n$-isometric dilation.
\end{prop}
\begin{proof}
Let $(T_1,\dots,T_{n-1},V)$ on $\mathcal K\supseteq \mathcal H$ be a
$\Gamma_n$-isometric dilation of $(S_1,\dots,S_{n-1},P)$. Let $\mathcal K_0$
be the space defined as
$$\mathcal K_0=\overline{\textup{span}}\{ T_1^{m_1}\dots T_{n-1}^{m_{n-1}}V^n h\,:\;
h\in\mathcal H \textup{ and }m_1,\dots,m_{n-1},n\in \mathbb N \cup \{0\}
\}.$$ Clearly $\mathcal K_0$ is invariant under $T_1^{m_1},\dots,
T_{n-1}^{m_{n-1}}$ and $V^n$, for any non-negative integer $m_1,\dots,m_{n-1}$ and
$n$. Therefore if we denote the restrictions of $T_1,\dots, T_{n-1}$ and $P$
to the common invariant subspace $\mathcal K_0$ by $T_{11},
\dots, T_{1{n-1}}$ and $V_1$ respectively, we get $T_{11}^{m_1}k=T_1^{m_1}k,
\,\dots ,  T_{1{n-1}}^{m_{n-1}}k=T_{n-1}^{m_{n-1}}k, \textup{ and } V_1^nk=V^nk,\quad
\textup{ for any }k\in\mathcal K_0.$ Hence
$$\mathcal K_0=\overline{\textup{span}}\{ T_{11}^{m_1}\dots T_{1{n-1}}^{m_{n-1}}V_1^n h\,:\;
h\in\mathcal H \textup{ and }m_1,m_2,n\in \mathbb N \cup \{0\} \}.
$$ Therefore for any non-negative integers $m_1,\dots, m_{n-1}$ and $n$ we have
$$ P_{\mathcal H}(T_{11}^{m_1}\dots T_{1{n-1}}^{m_{n-1}}V_1^{n})h=P_{\mathcal H}(T_1^{m_1}T_2^{m_2}V^n )h,
\quad \textup{ for all }h\in\mathcal H .
$$
Now $(T_{11},\dots , T_{1{n-1}},V_1)$ is a
$\Gamma_{n-1}$-contraction by being the restriction of a
$\Gamma_{n-1}$-contraction $(T_1,\dots , T_{n-1},V)$ to a common
invariant subspace $\mathcal K_0$. Again $V_1$, being the
restriction of an isometry to an invariant subspace, is also an
isometry. Therefore by Theorem \ref{G-isometry}-part (4),
$(T_{11},\dots ,T_{1{n-1}},V_1)$ is a $\Gamma_n$-isometry. Hence
$(T_{11},\dots ,T_{1{n-1}},V_1)$ is a minimal $\Gamma_n$-isometric
dilation of $(S_1,\dots ,S_{n-1},P)$.

\end{proof}

\begin{prop}\label{dilation-extension}
Let $(T_1,\dots ,T_{n-1},V)$ on $\mathcal K\supseteq \mathcal H$
be a $\Gamma_n$-isometric dilation of a $\Gamma_n$-contraction
$(S_1,\dots ,S_{n-1},P)$. If $(T_1,\dots ,T_{n-1},V)$ is minimal,
then $(T_1^*,\dots ,T_{n-1}^*,V^*)$ is a $\Gamma_n$-co-isometric
extension of $(S_1^*,\\ \dots ,S_{n-1}^*,P^*)$. Conversely, if
$(T_1^*,\dots ,T_{n-1}^*,V^*)$ is a $\Gamma_n$-co-isometric
extension of $(S_1^*,\dots ,S_{n-1}^*,P^*)$ then $(T_1,\dots
,T_{n-1},V)$ is a $\Gamma_n$-isometric dilation of $(S_1,\dots
,S_{n-1},P)$.
\end{prop}
\begin{proof}
We first prove that $S_1P_{\mathcal H}=P_{\mathcal H}T_1,
S_2P_{\mathcal H}=P_{\mathcal H}T_2$ and $PP_{\mathcal
H}=P_{\mathcal H}V$, where $P_{\mathcal H}:\mathcal K \rightarrow
\mathcal H$ is orthogonal projection onto $\mathcal H$. Clearly
$$\mathcal K=\overline{\textup{span}}\{ T_1^{m_1}\dots T_{n-1}^{m_{n-1}}V^n h\,:\;
h\in\mathcal H \textup{ and }m_1,\dots, m_{n-1},n\in \mathbb N
\cup \{0\} \}.$$ Now for $h\in\mathcal H$ we have that
\begin{align*}
S_iP_{\mathcal H}(T_1^{m_1}\dots T_{n-1}^{m_{n-1}}V^n h) &
=S_i(S_1^{m_1}\dots S_{n-1}^{m_{n-1}}P^n h) \\&=S_1^{m_1}\dots
S_i^{m_i+1}\dots S_{n-1}^{m_{n-1}}P^n h\\& =P_{\mathcal
H}(T_1^{m_1}\dots T_i^{m_i+1}T_{n-1}^{m_{n-1}}V^n h)\\&
=P_{\mathcal H}T_i(T_1^{m_1}\dots T_{n-1}^{m_{n-1}}V^n h).
\end{align*}
Thus we have $S_iP_{\mathcal H}=P_{\mathcal H}T_i$ for each $i$
and similarly one can prove that $PP_{\mathcal H}=P_{\mathcal
H}V$. Also for $h\in\mathcal H$ and $k\in\mathcal K$ we have that
\[
\langle S_i^*h,k \rangle =\langle P_{\mathcal H}S_i^*h,k \rangle
=\langle S_i^*h,P_{\mathcal H}k \rangle =\langle h,S_iP_{\mathcal
H}k \rangle =\langle h,P_{\mathcal H}T_i \rangle =\langle T_i^*h,k
\rangle .
\]
Hence $S_i^*=T_i^*|_{\mathcal H}$ and similarly
$P^*=V^*|_{\mathcal H}$. The converse part is obvious.

\end{proof}

\noindent We now present our dilation theorem.

\begin{thm} \label{isometric-dilation}
Let $(S_1,\dots,S_{n-1}P)$ be a $\Gamma_n$-contraction on a Hilbert space $\mathcal{H}$ and let $(A_1,\dots,A_{n-1})$ and $(B_1,\dots , B_{n-1})$ be the $\ft$-tuple of $(S_1, \dots , S_{n-1},P)$ and $(S_1^*, \dots , S_{n-1}^*,P^*)$ respectively. Let
$\mathcal{N}=\mathcal{H}\oplus\mathcal{D}_{p}\oplus\mathcal{D}_{p}\oplus\mathcal{D}_{p}\oplus\dots=\mathcal{H}\oplus
l^2(\mathcal{D}_{p})=\mathcal H \oplus l^2(\mathcal D_P)$ and consider the operators $T_1,\dots, T_{n-1},V$ defined on $\mathcal{N}$ by
\begin{align*}
T_i & = \begin{bmatrix} S&0&0&0&\dots\\
A_{n-i}^*{D}_{p}&A_i&0&0&\dots\\
0&A_{n-i}^*&A_i&0&\dots\\
0&0&A_{n-i}^*&A_i&\dots\\
\dots&\dots&\dots&\dots&\dots\\
\end{bmatrix}\,, \quad i=1,\dots,n-1 \quad \text{ and} \\
V &= \begin{bmatrix} P&0&0&0&\dots\\
{D}_{p}&0&0&0&\dots\\
0&I&0&0&\dots\\
0&0&I&0&\dots\\
\dots&\dots&\dots&\dots&\dots\\
\end{bmatrix}.
\end{align*}
If
\[
\left(
\frac{n-1}{n}(A_1+A_{n-1}^*z),\frac{n-2}{n}(A_2+A_{n-2}^*z),\dots,
\dfrac{1}{n}(A_{n-1}+A_1^*z) \right)
\]
and
\[
\left(
\frac{n-1}{n}(B_{1}^*+B_{n-1}z),\frac{n-2}{n}(B_{2}^*+B_{n-2}z),\dots,
\dfrac{1}{n}(B_{n-1}^*+B_{1}z) \right)
\] are $\Gamma_{n-1}$-contractions for every $z\in \mathbb D$, then
\begin{enumerate}
\item $(T_1,\dots,T_{n-1},V)$ is a minimal $\Gamma_n$-isometric dilation of $(S_1,\dots,S_{n-1}P)$.
\item If $(\widehat{T}_1,\dots,\widehat{T}_{n-1},V)$ on $\mathcal{N}$ is a $\Gamma$-isometric dilation of $(S_1,\dots,S_{n-1}P)$, then
$\widehat{T}_i=T_i$ for $i=1,\dots ,n-1$.
\item If $(\overline{T}_1,\dots, \overline{T}_{n-1},\overline{V})$ is a $\Gamma_n$-isometric dilation of $(S_1,\dots,S_{n-1}P)$ where
$\overline{V}$ is a minimal isometric dilation of $P$, then $(\overline{T}_1,\dots, \overline{T}_{n-1},\overline{V})$ is unitarily equivalent to $(T_1,\dots, T_{n-1},V)$.
\end{enumerate}
Thus $(2)$ and $(3)$ guarantee the uniqueness of the
$\Gamma_n$-isometric dilation $(T_1,\dots, T_{n-1},V)$ of $(S_1,\dots,S_{n-1},P)$ when $V$ is the minimal
isometric dilation of $P$ on $\mathcal N$.
\end{thm}
 
\begin{proof}
 
If the $\ft$-tuples of the given $\Gamma_n$-contraction and its adjoint satisfy the given conditions, then by Theorem \ref{main-dilation-theorem}, $(R_1,\dots, R_{n-1},U)$ on $\mathcal K$ (as in Theorem \ref{main-dilation-theorem}) is a minimal $\Gamma_n$-unitary dilation of $(S_1,\dots, S_{n-1},P)$. It is evident that
$\mathcal N=\mathcal H\oplus l^2({\mathcal
 D_P})= H\oplus \mathcal
D_P\oplus\mathcal D_P\oplus\cdots$ is a joint invariant subspace
of $R_1,\dots, R_{n-1}, U$ and that $R_i|_{\mathcal N}=T_i$ for $i=1,\dots, n-1$ and $U|_{\mathcal N}=V$. Therefore, by definition $(T_1,\dots, T_{n-1},V)$ is a
$\Gamma_n$-isometry. It is obvious from the matrices of $T_1\dots, ,T_{n-1}, V$ that the restriction of the $\Gamma_n$-co-isometry $(T_1^*,\dots, T_{n-1}^*,V^*)$ to the joint invariant subspace $\mathcal H$ of $T_1^*,\dots, T_{n-1}^*, V^*$ is the $\Gamma_n$-contraction $(S_1^*,\dots, S_{n-1}^*,P^*)$. Therefore by Proposition
\ref{dilation-extension}, $(T_1,\dots, T_{n-1},V)$ is a $\Gamma_n$-isometric dilation of $(S_1,\dots, S_{n-1},P)$. The minimality of this $\Gamma_n$-isometric dilation follows from the fact that $\mathcal N$ and $V$ are respectively the minimal isometric dilation space and minimal isometric dilation of $P$. This completes the proof of part-(1) of the theorem. The proofs of part-(2) and part-(3) are similar to the corresponding parts of Theorem \ref{main-dilation-theorem} and we skip them. Hence the proof is
complete.
\end{proof}

The following result is a corollary of the previous theorem and it provides a set of sufficient conditions under which a commuting operator tuple $(S_1, \dots , S_{n-1},P)$ becomes a $\G$-contraction.

\begin{thm}\label{sufficient1} Let $S_1,\dots, S_{n-1},P$ be commuting
operators on a Hilbert space $\mathcal H$ with $\|S_i\|\leq {n \choose i}$ for each $i$ and $\|P\|\leq 1$. If there are operators $A_1,\dots, A_{n-1}$ in $\mathcal B(\mathcal D_P)$ and $B_1,\dots, B_{n-1}$ in $\mathcal B(\mathcal D_{P^*})$ with $\omega(A_i+ A_{n-i}z)\leq {n \choose i}$ and $\omega(B_i+B_{n-i}z)\leq {n \choose i}$ respectively for $i=1,\dots, n-1$ and for all $z\in {\mathbb D}$ such that
\begin{enumerate}

\item $\sigma (S_1,\dots, S_{n-1},P) \subseteq \Gamma_n$ ;

\item $S_i-S_{n-i}^*P=D_PA_iD_P$ for $i=1,\dots, n-1$ ;

\item $\left(
\dfrac{n-1}{n}(A_1+A_{n-1}^*z),\dfrac{n-2}{n}(A_2+A_{n-2}^*z),\dots,
\dfrac{1}{n}(A_{n-1}+A_1^*z) \right)$
and
$$
\left(
\frac{n-1}{n}(B_{1}^*+B_{n-1}z),\frac{n-2}{n}(B_{2}^*+B_{n-2}z),\dots,
\dfrac{1}{n}(B_{n-1}^*+B_{1}z) \right)
$$
\end{enumerate}
are
$\Gamma_{n-1}$-contraction for every $z\in \mathbb T$,
then $(S_1,\dots, S_{n-1},P)$ is a $\Gamma_n$-contraction.

\end{thm}

\begin{proof}
If $S_1,\dots, S_{n-1},P$ and $A_i,B_i$ satisfy the given conditions, then we can construct operators $R_1,\dots, R_{n-1},U$ on $\mathcal K$ as in Theorem \ref{main-dilation-theorem} such that $(R_1,\dots, R_{n-1},U)$ is a $\Gamma_n$-unitary dilation of $(S_1,\dots, S_{n-1},P)$. Consequently taking restriction of $(R_1,\dots, R_{n-1},U)$ to the joint invariant subspace $\mathcal N$ we obtain $(T_1,\dots, T_{n-1},V)$ as in Theorem \ref{isometric-dilation}, which is a $\Gamma_n$-isometric dilation of $(S_1,\dots, S_{n-1},P)$. Also, the block matrices of $T_1,\dots, T_{n-1},V$ shows that $(S_1^*,\dots, S_{n-1}^*,P^*)$ is the restriction of the $\Gamma_n$-contraction $(T_1^*,\dots, T_{n-1}^*,V^*)$ to the joint invariant subspace $\mathcal H$ of $T_1^*,\dots, T_{n-1}^*,V^*$. Thus, $(S_1^*,\dots, S_{n-1}^*,P^*)$ and hence $(S_1,\dots, S_{n-1},P)$ is a $\Gamma_n$-contraction and the proof is complete.

\end{proof}

\subsection{An operator model for a class of
$\Gamma_n$-contractions}\label{functional-model}

We start with an elementary useful result whose proof is a routine exercise.
\begin{prop}\label{easyprop1}
If $T$ is a contraction and $V$ is its minimal isometric dilation
then $T^*$ and $V^*$ have defect spaces of same dimension.
\end{prop}

We achieve the following operator model as a consequence of the minimal $\G$-isometric dilation as in Theorem \ref{isometric-dilation}. 

\begin{thm}\label{model2}
 Let $(S_1,\dots ,S_{n-1},P)$ be a $\Gamma_n$-contraction acting on $\mathcal H$ and let $(A_1,\dots,A_{n-1})$ and $(B_{1},\dots,B_{{n-1}})$ be the $\ft$-tuples of $(S_1,\dots,S_{n-1},P)$ and $(S_1^*,\dots ,S_{n-1}^*,P^*)$ respectively. Suppose
 \[
 \left(
\frac{n-1}{n}(A_1+A_{n-1}^*z),\frac{n-2}{n}(A_2+A_{n-2}^*z),\dots,
\dfrac{1}{n}(A_{n-1}+A_1^*z) \right)
\]
and
\[
\left(
\frac{n-1}{n}(B_{1}^*+B_{n-1}z),\frac{n-2}{n}(B_{2}^*+B_{n-2}z),\dots,
\dfrac{1}{n}(B_{n-1}^*+B_{1}z) \right)
\]
are $\Gamma_{n-1}$-contraction for every $z\in \mathbb T$. Let $(\widehat T_1,\dots , \widehat T_{n-1},\widehat V)$ on $\mathcal N_*=\mathcal H\oplus \mathcal D_{P^*}\oplus\mathcal D_{P^*}\oplus
 \hdots$ be given by
 \[
 \widehat T_i=\begin{bmatrix}
 S_1 & D_{P^*}{B_{n-i}} & 0 & 0 & \cdots \\
 0 & {B_{i}}^* & {B_{n-i}} & 0 & \cdots \\
 0 & 0 & {B_{i}}^* & {B_{n-i}} & \cdots \\
 0 & 0 & 0 & {B_{i}}^* & \cdots \\
 \vdots & \vdots & \vdots & \vdots & \ddots
 \end{bmatrix}
 \;,\;
 \widehat V=\begin{bmatrix}
P & D_{P^*} & 0 & 0 & \cdots\\
0 & 0 & I & 0 & \cdots\\
0 & 0 & 0 & I & \cdots \\
0 & 0 & 0 & 0 & \cdots\\
\vdots & \vdots & \vdots & \vdots & \ddots
\end{bmatrix}\,.
\]
Then
 \begin{enumerate}
 \item $(\widehat T_1,\dots ,\widehat T_{n-1},\widehat V)$ is a $\Gamma_n$-co-isometry for which $\mathcal H$ is a common invariant subspace and $(\widehat T_1|_{\HS},\dots ,\widehat T_{n-1}|_{\HS},\widehat V|_{\HS})=(S_1, \dots , S_{n-1},P)$;
 
 \item there is an orthogonal decomposition $\mathcal N_*=\mathcal N_1\oplus \mathcal N_2$ into reducing subspaces for each $\widehat T_i$ and $\widehat V$ such that the operator tuple $(\widehat T_1|_{\mathcal N_1},\dots ,\widehat T_{n-1}|_{\mathcal N_1},\widehat V|_{\mathcal N_1})$ is a $\Gamma_n$-unitary and the counter part $(\widehat T_1|_{\mathcal N_2}, \dots , \widehat T_{n-1}|_{\mathcal N_2},\widehat V|_{\mathcal N_2})$ is a  pure $\Gamma_n$-co-isometry;
 
 \item $\mathcal N_2$ can be identified with $H^2(\mathcal D_{\widehat V})$, where $D_{\widehat V}$ has same dimension as of $\mathcal D_P$. The operators $\widehat T_i|_{\mathcal N_2}$, $i=1,\dots ,n-1$ and $\widehat V|_{\mathcal N_2}$ are respectively unitarily equivalent to $T_{A_i+A_{n-i}^*\bar z}$ and $T_{\bar
 z}$ defined on $H^2(\mathcal D_{\widehat V})$.
 \end{enumerate}
 \end{thm}
 \begin{proof}
 Since the $\ft$-tuples $(A_1,\dots, A_{n-1})$ and $({B_1},\dots ,{F_{n-1}})$ satisfy the given conditions, by Theorem \ref{isometric-dilation}, we have that $({\widehat T_1}^*,\dots ,{\widehat T_{n-1}}^*,{\widehat V}^*)$
 is minimal $\Gamma_n$-isometric dilation of $(S_1^*,\dots, S_{n-1}^*,P^*)$,
 where ${\hat V}^*$ is the minimal isometric dilation of $P^*$.
 Therefore, it follows from Proposition \ref{dilation-extension} that $(\widehat T_1,\dots , \hat T_{n-1},\hat V)$ is $\Gamma_n$-co-isometric extension of $(S_1,\dots,S_{n-1},P)$. So we have that $\mathcal H$ is a
 common invariant subspace of $\widehat T_i$, $i=1,\dots, n-1$ and $\widehat V$ and that
\[
\widehat T_i|_{\mathcal H}=S_i\,,i=1,\dots, n-1 \quad \& \quad \hat V|_{\mathcal H}=P.
\]
Again since $({\widehat T_1}^*,\dots,{\widehat T_{n-1}}^*,\hat V^*)$ is a $\Gamma_n$-isometry, by Wold decomposition (see Theorem \ref{G-isometry}, part-(4)), there is an orthogonal decomposition $\mathcal N_*=\mathcal N_1\oplus\mathcal N_2$ into
 reducing subspaces of $\widehat T_i$, $i=1,\dots,n-1$ and $\widehat V$ such that
 $(\widehat T_i|_{\mathcal N_i},\dots,\widehat T_{n-1}|_{\mathcal N_1},\widehat V|_{\mathcal N_1})$ is a
 $\Gamma_n$-unitary and $(\widehat T_1|_{\mathcal N_2},\dots,\widehat T_{n-1}|_{\mathcal N_2},\widehat V|_{\mathcal N_2})$
 is a pure $\Gamma_n$-co-isometry. If we denote
 $(\widehat T_1|_{\mathcal N_1},\dots,\hat T_{n-1}|_{\mathcal N_1},\hat V|_{\mathcal N_1})
 = (T_{11},\dots,T_{1\;n-1},V_1)$ and $(\widehat T_1|_{\mathcal N_2},\dots, \widehat T_{n-1}|_{\mathcal N_2},\hat V|_{\mathcal N_2})
 = (T_{21},\dots ,T_{2\;n-1},V_2)$ then with respect to the orthogonal decomposition $\mathcal K_*=\mathcal K_1\oplus \mathcal K_2$
 we have
 \[
 \widehat T_1=
 \begin{bmatrix}
 T_{11} & 0 \\
 0 & T_{21}
 \end{bmatrix}\;, \dots,\;
 \hat T_{n-1}=
 \begin{bmatrix}
 T_{1\;n-1} & 0 \\
 0 & T_{2\; n-1}
 \end{bmatrix}\;
 \text{ and }
 \hat V=
 \begin{bmatrix}
 V_1 & 0 \\
 0 & V_2
 \end{bmatrix}.
 \]
The fundamental equations
\[
\widehat T_i-{\widehat T_{n-1}}^*\widehat V =D_{\widehat V}X_iD_{\widehat V}\,, \quad i=1,\dots, n-1
\]
of $(\widehat T_{1},\dots,\widehat T_{n-1},\hat V)$ clearly become
 \[
 \begin{bmatrix}
 T_{1i}-T_{1\;n-i}^*V_1 & 0 \\
 0 & T_{2i}-T_{2\;n-i}^*V_2
 \end{bmatrix}
 =\begin{bmatrix}
 0 & 0 \\
 0 & D_{V_2}X_{1\;n-i}D_{V_2}
 \end{bmatrix},\;\;
 X=\begin{bmatrix} X_{11}\\ \vdots \\ X_{1\;n-1}\end{bmatrix}.
 \]
 Since $\mathcal D_{\widehat V}=\mathcal D_{V_2}$, $(\widehat T_1,\dots,\hat T_{n-1},\widehat V)$ and $(T_{21},\dots,T_{2\;{n-1}},V_2)$ have same $\ft$-tuples. Now we apply Theorem \ref{model1} to $(T_{21}^*,\dots,T_{2\;n-1}^*,V_2^*)$, which is a pure $\Gamma_n$-isometry, and get the following:
 \begin{enumerate}
 \item $\mathcal N_2$ can be identified with $H^2(\mathcal D_{V_2})=H^2(\mathcal
 D_{\widehat V})$;
 \item The operators $T_{21}^*, \dots, T_{2\;n-1}^*$ and $V_2^*$ are unitarily equivalent to the Toeplitz operators $T_{A_1^*+A_{n-1} z}, \dots, T_{A_{n-1}^*+A_1 z}$ and $T_z$ respectively acting on $H^2(\mathcal D_{\widehat V})$.
 \end{enumerate}
We explain the validity of part-(2) above. Since $\widehat V^*$ is the minimal
 isometric dilation of $P^*$ by Proposition \ref{easyprop1}, $\mathcal D_{\widehat V}$ and $\mathcal D_P$ have same dimension. An elementary calculation guarantees that $(A_1,\dots, A_{n-1})$ is the $\ft$-tuple of $(\widehat T_1,\dots, \widehat T_{n-1}, \widehat V)$.
 Therefore, $(\widehat T_1|_{\mathcal N_2},\dots,\widehat T_{n-1}|_{\mathcal N_2},\widehat V|_{\mathcal
 N_2})$ is unitarily equivalent to $(T_{A_1+A_{n-1}^*\bar z},\dots, T_{A_{n-1}+A_1^*\bar z},T_{\bar
 z})$ defined on $H^2(\mathcal D_{\widehat V})$. The proof is now complete.
 
\end{proof}

\vspace{0.4cm}

\section{Interplay between rational dilation and distinguished varieties in the symmetrized polydisc }

\vspace{0.4cm}

\noindent Recall that a distinguished variety in $\gn$ is a set of the form $W=W'\cap \gn$, where $W'$ is a complex affine variety in $\C^n$ such that $W'\cap \partial \G \subseteq b\G$. In this Section, we shall show how the existence of a $\G$-unitary or $\G$-isometric dilation of a $\G$-contraction $(S_1, \dots , S_{n-1},P)$ on the minimal dilation space of $P$ guarantees the existence of distinguished varieties in $\gn$ determined by the $\ft$-tuples and vice versa when the defect space of $P$ or $P^*$ is finite dimensional. As a consequence of this interplay (between rational dilation and distinguished varieties), we obtain a new characterization for the distinguished varieties in $\gn$. We begin with the following result which represents an irreducible distinguished variety in $\gn$ in terms of Taylor joint spectrum of $n-1$ commuting matrix pencils. We borrow this result from an earlier work of the author \cite{spal2}.

\begin{thm}  \label{thm:DVchar-1}
Let
\begin{equation}\label{eq:W1}
\Lambda = \{ (s_1,\dots,s_{n-1},p)\in \mathbb G_n \,: \, (s_1,\dots,s_{n-1}) \in \sigma_T(F_1^*+pF_{n-1}\,,\,
F_2^*+pF_{n-2}\,,\,\dots\,, F_{n-1}^*+pF_1) \},
\end{equation}
where $F_1,\dots,F_{n-1}$ are complex square matrices of same order satisfying the following conditions:
\begin{itemize}
\item[(i)] $[F_i,F_j]=0$ and $[F_i^*,F_{n-j}]=[F_j^*,F_{n-i}]$, for $1\leq i<j\leq
n-1$ ; \item[(ii)] for every $z\in \D$ the Taylot joint spectrum of $(F_1^*+zF_{n-1}, F_2^*+zF_{n-2}, \dots , F_{n-1}^*+zF_1, zI)$ is contained in $\gn$ ;
\item[(iii)] the polynomials $\{ f_1,\dots ,f_{n-1} \}$, where $f_i=\det\,(F_i^*+pF_{n-i}-s_iI)$ for each $i$, form a regular sequence ;
\item[(iv)] the complex algebraic set generated by the polynomials $S=\{ f_1,\dots, f_{n-1} \}$ is irreducible.
\end{itemize}
Then, $\Lambda$ is a distinguished variety in $\mathbb G_n$. Furthermore, $\Lambda$ is a part of an affine algebraic curve which is a set-theoretic complete intersection.

Conversely, every distinguished variety $\Lambda$ in $\mathbb G_n$ is a part of an affine algebraic curve lying in $\gn$ which is a complete intersection and has representation as in $($\ref{eq:W1}$)$, where $F_1,\dots, F_{n-1}$ are complex square matrices of same order satisfying the above conditions ${(i)-(iv)}$.

\end{thm}
Let us recall from Section 1 that a set of $n-1$ complex square matrices of same order $\{ F_1, \dots , F_{n-1} \}$ is said to define a distinguished variety $W=W'\cap \gn$ in the symmetrized polydisc, where $W'$ is an affine variety in $\C^n$, if
\[
W = \{ (s_1,\dots,s_{n-1},p)\in \mathbb G_n \,: \, (s_1,\dots,s_{n-1}) \in \sigma_T(F_1^*+pF_{n-1}\,,\,
F_2^*+pF_{n-2}\,,\,\dots\,, F_{n-1}^*+pF_1) \},
\]
so that $\{ \det \, (F_i^*+z_n F_{n-i}-z_iI)\,:\,1 \leq i \leq n-1 \}$ becomes a set of generators for $W'$.

First we deal with dilation of a pure $\G$-contraction $(S_1, \dots , S_{n-1},P)$, i.e. a $\G$-contraction with $P$ being a pure contraction. The purity of $P$ or $P^*$ leads to a lot of interesting outcomes as we have witnessed in \cite{sourav6, spal2}. Here we achieve our results along with the bonus that the dilation of a pure $\G$-contraction on the minimal dilation space of $P$ is equivalent to its dilation on a distinguished variety determined by the $\ft$-tuple of $(S_1^*, \dots , S_{n-1}^*, P^*)$. The mystery will be unfolded after a couple of dilation theoretic results.

\begin{thm} [\cite{spal2}, Theorem 5.9]   \label{thm:VN}

Let $\Sigma=(S_1,\dots,S_{n-1},P)$ be a $\Gamma_n$-contraction with $\ft$-tuple $(F_1,\dots,F_{n-1})$ such that $P^*$ is pure and $\dim \mathcal D_{P} < \infty $. If $(F_1,\dots,F_{n-1})$ defines a distinguished variety $\Lambda_{\Sigma}$ in $\gn$, then both $(S_1, \dots , S_{n-1},P)$ and $(S_1^*,\dots,S_{n-1}^*,P^*)$ admit normal $\partial \ov{\Lambda}_{\Sigma}-$dilation, where $\partial \ov{\Lambda}_{\Sigma}= \ov{\Lambda}_{\Sigma} \setminus \Lambda_{\Sigma}\;(=b\Gamma_n \cap \ov{\Lambda}_{\Sigma}) $. Moreover, the dilation of $(S_1^*, \dots ,S_{n-1}^*,P^*)$ is minimal and acts on the minimal unitary dilation space of $P^*$.

\end{thm}

\begin{thm} \label{thm:pure-dilation}
Let $(S_1, \dots , S_{n-1},P)$ on $\HS$ be a pure $\G$-contraction and let $(B_1, \dots , B_{n-1})$ be the $\ft$-tuple of $(S_1^*, \dots , S_{n-1}^*,P^*)$. Then the following are equivalent.
\begin{enumerate}
\item $(S_1, \dots , S_{n-1},P)$ dilates to a $\G$-unitary $(R_1, \dots , R_{n-1},U)$ on $L^2(\mathcal D_{P^*})$ with $U$ being the minimal unitary dilation of $P$.

\item $\left(
\frac{n-1}{n}(B_1^*+B_{n-1}z),\frac{n-2}{n}(B_2^*+B_{n-2}z),\dots,
\frac{1}{n}(B_{n-1}^*+B_1z) \right)
$
is a $\Gamma_{n-1}$-contraction for every $z\in \T$.

\end{enumerate}

\end{thm}

\begin{proof}

\textbf{(1) $\Rightarrow$ (2)} Since $P$ is a pure contraction, it follows from Nagy-Foias theory (see \cite{nagy}) that upto a unitary $M_z$ on $L^2(\mathcal D_{P^*})$ is the minimal unitary dilation of $P$. Thus, it follows from Theorem \ref{main-dilation-theorem} that
\[
\left(
\frac{n-1}{n}(B_1^*+B_{n-1}z),\frac{n-2}{n}(B_2^*+B_{n-2}z),\dots,
\frac{1}{n}(B_{n-1}^*+B_1z) \right)
\]
is a $\Gamma_{n-1}$-contraction for every $z\in \T$.\\

\noindent \textbf{(2) $\Rightarrow$ (1)} Suppose (2) holds. Then by Theorem 4.3 in \cite{sourav6}, $(T_{B_1^*+B_{n-1}z}, \dots , T_{B_{n-1}^*+B_{1}z},T_z)$ on $H^2(\mathcal D_{P^*})$ is a $\G$-isometric dilation of $(S_1, \dots , S_{n-1},P)$. It is obvious that the $\G$-isometry $(T_{B_1^*+B_{n-1}z}, \dots , T_{B_{n-1}^*+B_{1}z},T_z)$ extends to the $\G$-unitary $(M_{B_1^*+B_{n-1}z}, \dots , M_{B_{n-1}^*+B_{1}z},M_z)$ on $L^2(\mathcal D_{P^*})$. Therefore, $(M_{B_1^*+B_{n-1}z}, \dots , M_{B_{n-1}^*+B_{1}z},M_z)$ is a $\G$-unitary dilation of $(S_1, \dots , S_{n-1}, P)$.

\end{proof}

The following result provides our first desired interplay between dilation and distinguished varieties under the purity assumption.

\begin{thm} \label{thm:dilation-variety1}

Let $(S_1, \dots , S_{n-1},P)$ be a pure $\Gamma_n$-contraction on $\HS$ with $(B_1, \dots , B_{n-1})$ being the $\ft$-tuple of $(S_1^*, \dots , S_{n-1}^*,P^*)$. If $\mathcal D_{P^*}$ is finite dimensional then the following are equivalent.

\begin{enumerate}

\item For every $z\in \D$, $\sigma_T(B_1^*+zB_{n-1}, B_2^*+zB_{n-2}, \dots , B_{n-1}^*+zB_1, zI)\subseteq \gn$ and $(S_1, \dots , S_{n-1},P)$ possesses a $\Gamma_n$-unitary dilation on $L^2(\mathcal D_{P^*})$.\\

\item For every $z\in \D$, $\sigma_T(B_1^*+zB_{n-1}, B_2^*+zB_{n-2}, \dots , B_{n-1}^*+zB_1, zI)\subseteq \gn$ and the tuple
$
\left(
\frac{n-1}{n}(B_1^*+B_{n-1}z),\frac{n-2}{n}(B_2^*+B_{n-2}z),\dots,
\frac{1}{n}(B_{n-1}^*+B_1z) \right)
$
is a $\Gamma_{n-1}$-contraction for every $z\in \T$.\\

\item The set of polynomials $\{f_i=\det\, (B_i^*+pB_{n-i}-s_iI):\, 1 \leq i \leq n-1 \}$ defines a distinguished variety $\Lambda$ in $\gn$.\\

\item For every $z\in \D$, $\sigma_T(B_1^*+zB_{n-1}, B_2^*+zB_{n-2}, \dots , B_{n-1}^*+zB_1, zI)\subseteq \gn$ and $(S_1, \dots , S_{n-1},P)$ possesses a normal $\partial \ov{\Lambda}-$dilation on $L^2(\mathcal D_{P^*})$, where $\Lambda = \mathcal Z(f_1, \dots , f_{n-1})\cap \gn$ and $\partial \ov{\Lambda}_{\Sigma}=b\Gamma_n \cap \ov{\Lambda}_{\Sigma} $.

\end{enumerate}

\end{thm}

\begin{proof}

\textbf{(1) $\Rightarrow$ (2)} This follows from Theorem \ref{thm:pure-dilation}.\\

\noindent \textbf{(2) $\Rightarrow$ (3)} Suppose
$
\left(
\frac{n-1}{n}(B_1^*+B_{n-1}z),\frac{n-2}{n}(B_2^*+B_{n-2}z),\dots,
\frac{1}{n}(B_{n-1}^*+B_1z) \right)
$
is a $\Gamma_{n-1}$-contraction for every $z\in \T$.  Thus, by the commutativity of the above tuple we have that $[B_i,B_j]=0$ and $[B_i^*, B_{n-j}]=[B_j^*, B_{n-i}]$. Since $\mathcal D_{P^*}$ is finite dimensional, $B_1, \dots , B_{n-1}$ are commuting matrices. It can be easily verified that $B_1^*+B_{n-1}z,\, B_2^*+B_{n-2}z, \dots , B_{n-1}^*+B_1z$ are commuting normal operators for every $z\in \T$. Since $\sigma_T(B_1^*+B_{n-1}z, \dots , B_{n-1}^*+B_1z,zI)\subset \mathbb G_n$ for every $z\in \D$, it follows that the variety $V$ generated by the polynomials $\{f_i=\det\, (B_i^*+pB_{n-i}-s_iI):\, 1 \leq i \leq n-1 \}$ intersects $\gn$. This variety $V$ cannot contain a point $(s_1, \dots , s_{n-1},p)\in \partial \Gamma_n \setminus b\Gamma_n$. Indeed, for any point $(s_1, \dots , s_{n-1},p)\in \partial \Gamma_n \setminus b\Gamma_n$, we must have $p\in \D$ (see Theorem \ref{thm:DB}) and by hypothesis we have that $\sigma_T(B_1^*+B_{n-1}z, \dots , B_{n-1}^*+B_1z,zI)\subset \mathbb G_n$ for every $z\in \D$. Thus, $V$ exits through $b\Gamma_n$ and consequently
\[
\Lambda = \{ (s_1,\dots,s_{n-1},p)\in \mathbb G_n \,:
\; (s_1,\dots,s_{n-1}) \in \sigma_T(B_1^*+pB_{n-1}\,,\,
B_2^*+pB_{n-2}\,,\,\dots\,, B_{n-1}^*+pB_1) \}
\]
is a distinguished variety in $\gn$ defined by the polynomials $\{f_1, \dots ,f_{n-1} \}$. This distinguished variety may or may not be irreducible.\\

\noindent \textbf{(3) $\Rightarrow$ (4)} If the set of polynomials $\{f_i=\det\, (B_i^*+pB_{n-i}-s_iI):\, 1 \leq i \leq n-1 \}$ defines a distinguished variety $\Lambda$ in $\gn$, then it follows from Theorem \ref{thm:DVchar-1} that $\sigma_T(B_1^*+B_{n-1}z, \dots , B_{n-1}^*+B_1z,zI)\subset \mathbb G_n$ for every $z\in \D$. The rest part follows from Theorem \ref{thm:VN}.\\

\noindent \textbf{(4) $\Rightarrow$ (1)} Note that a normal $\partial \ov{\Lambda}-$dilation is a $\G$-unitary dilation, because, $\partial \ov{\Lambda} \subseteq b\G$. Thus, (4) $\Rightarrow$ (1) follows trivially.

\end{proof}

The following result is a next step to Theorem \ref{thm:dilation-variety1} in the sense that here we remove the purity condition on $P$ and present a more generalized interplay between rational dilation and distinguished varieties.

\begin{thm} \label{thm:dilation-variety2}

Let $(S_1, \dots , S_{n-1},P)$ be a $\Gamma_n$-contraction on $\HS$ with $\mathcal D_P$ and $\mathcal D_{P^*} $ being finite dimensional. Suppose $A=(A_1, \dots , A_{n-1})$ and $B=(B_1, \dots , B_{n-1})$ are the $\ft$-tuples of $(S_1, \dots , S_{n-1}, P)$ and $(S_1^*, \dots , S_{n-1}^*, P^*)$ respectively. Let us denote $A(z)=(A_1+A_{n-1}^*z, \dots , A_{n-1}+A_1^*z,zI)$ and $B(z)=\sigma_T(B_1^*+B_{n-1}z, \dots , B_{n-1}^*+B_1z,zI)$. Then the following are equivalent.

\begin{enumerate}

\item $\sigma_T(A(z)), \sigma_T(B(z)) \subseteq \mathbb G_{n}$ and $(S_1, \dots , S_{n-1},P)$ admits $\Gamma_n$-unitary dilation on $l^2(\mathcal D_{P^*})\oplus \HS \oplus l^2(\mathcal D_P)$.\\

\item $\sigma_T(A(z)), \sigma_T(B(z)) \subseteq \mathbb G_{n}$ and both
$
\left(
\frac{n-1}{n}(B_1^*+B_{n-1}z),\frac{n-2}{n}(B_2^*+B_{n-2}z),\dots,
\frac{1}{n}(B_{n-1}^*+B_1z) \right)
$
and
$
\left(
\frac{n-1}{n}(A_1+A_{n-1}^*z),\frac{n-2}{n}(A_2+A_{n-2}^*z),\dots,
\frac{1}{n}(A_{n-1}+A_1^*z) \right)
$
are $\Gamma_{n-1}$-contractions for every $z\in \T$.\\

\item Each of the sets of polynomials $\{g_i=\det\, (A_i+pA_{n-i}^*-s_iI):\, 1 \leq i \leq n-1 \}$ and $\{f_i=\det\, (B_i^*+pB_{n-i}-s_iI):\, 1 \leq i \leq n-1 \}$ defines a distinguished variety in the symmetrized polydisc.

\end{enumerate}

\end{thm}

\begin{proof}

Evidently \textbf{(1) $\Leftrightarrow$ (2)} follows from Theorem \ref{main-dilation-theorem} and \textbf{(2) $\Leftrightarrow$ (3)} follows from Theorem \ref{thm:dilation-variety1}.

\end{proof}

We conclude this Section with the following new characterization for the distinguished varieties in the symmetrized polydisc.

\begin{thm} \label{thm:DVchar-4}

Let $F_1, \dots , F_{n-1}$ be complex matrices of same order say $d$. Then the following are equivalent.
\begin{enumerate}
\item The set of polynomials $\{\det\,(F_i^*+pF_{n-i}-s_iI):\, 1 \leq i \leq n-1  \}$ defines a distinguished variety in $\gn$.

\item $
\Upsilon(z)=\left(
\frac{n-1}{n}(F_1^*+F_{n-1}z),\frac{n-2}{n}(F_2^*+F_{n-2}z),\dots,
\frac{1}{n}(F_{n-1}^*+F_1z) \right)
$
is a $\Gamma_{n-1}$-contraction for every $z\in \T$ and $\sigma_T(F_1^*+F_{n-1}z, \dots , F_{n-1}^*+F_1z,zI) \subseteq \gn$ for every $z\in \D$.

\end{enumerate}

\end{thm}

\begin{proof}

\textbf{(1) $\Rightarrow$ (2)} If set of polynomials $\{\det\,(F_i^*+pF_{n-i}-s_iI):\, 1 \leq i \leq n-1  \}$ defines a distinguished variety in $\gn$, then by Theorem \ref{thm:DVchar-1} we have that $\sigma_T(F_1^*+F_{n-1}z, \dots , F_{n-1}^*+F_1z,zI)\subset \mathbb G_n$ for every $z\in \D$. Also, it is evident that $(F_1^*+F_{n-1}z, \dots , F_{n-1}^*+F_1z,zI)$ is a commuting normal tuple for any $z\in \T$ and it follows from the hypothesis that $\sigma_T(F_1^*+F_{n-1}z, \dots , F_{n-1}^*+F_1z,zI) \subseteq b\G$ whenever $z$ is of unit modulus. Therefore, $(F_1^*+F_{n-1}z, \dots , F_{n-1}^*+F_1z,zI)$ is a $\G$-unitary and it follows from Lemma \ref{lem:BSR2} that
\[
\Upsilon(z) = \left(
\frac{n-1}{n}(F_1^*+F_{n-1}z),\frac{n-2}{n}(F_2^*+F_{n-2}z),\dots,
\frac{1}{n}(F_{n-1}^*+F_1z) \right)
\]
is a $\Gamma_{n-1}$-contraction for every $z\in \T$.\\

\noindent \textbf{(2) $\Rightarrow$ (1)} Suppose (2) holds. Since $\Upsilon(z)$ is a $\Gamma_{n-1}$-contraction for every $z\in \T$, it follows from Theorem \ref{model1} that $(T_{\phi_1}^*, \dots , T_{\phi_{n-1}}^*,T_z^*)$ is a $\G$-co-isometry on $H^2(\C^d)$, where $\phi_i(z)=F_i^*+F_{n-i}z$ for each $i$. A simple calculation (see Remark \ref{partial-converse}) shows that $(F_1, \dots , F_{n-1})$ is the $\ft$-tuple of $(T_{\phi_1}^*, \dots , T_{\phi_{n-1}}^*,T_z^*)$. Thus, $(F_1, \dots , F_{n-1})$ is the $\ft$-tuple of a $\G$-contraction $(T_{\phi_1}^*, \dots , T_{\phi_{n-1}}^*,T_z^*)$ whose adjoint, $(T_{\phi_1}, \dots , T_{\phi_{n-1}},T_z)$ is a pure $\G$-contraction, in particular a pure $\G$-isometry with $\mathcal D_{T_z^*}(=\C^{d^2})$ being finite dimensional. Also, $\sigma_T(F_1^*+F_{n-1}z, \dots , F_{n-1}^*+F_1z,zI) \subseteq \gn$ for every $z\in \D$. Therefore, it follows from Theorem \ref{thm:dilation-variety1} that the set of polynomials $\{\det\,(F_i+pF_{n-i}^*-s_iI):\, 1 \leq i \leq n-1  \}$ defines a distinguished variety $\Lambda$ in $\gn$. Needless to mention that
\[
\Lambda = \{ (s_1,\dots,s_{n-1},p)\in \mathbb G_n \,:
\; (s_1,\dots,s_{n-1}) \in \sigma_T(F_1+pF_{n-1^*}\,,\,
F_2+pF_{n-2^*}\,,\,\dots\,, F_{n-1}+pF_1^*) \}.
\]
The proof is now complete.

\end{proof}


\section{Appendix}

\vspace{0.4cm}

\noindent \textbf{\textit{Proof of Lemma} \ref{funda-properties}.}
\textbf{(1).} It suffices if we show $PA_i={B_i}^*P|_{\mathcal
D_{P}}$ because proofs to the other identities are similar. For
$D_{P}h \in \mathcal D_{P}$ and $D_{P^*}h^{\prime}\in \mathcal
D_{P^*}$, we have by virtue of (\ref{funda-repeat1}),
\begin{align*}
\langle PA_iD_{P}h,D_{P^*}h^{\prime} \rangle =\langle
D_{P^*}PA_iD_{P}h,h^{\prime} \rangle =\langle
PD_{P}A_iD_{P}h,h^{\prime}\rangle &=\langle
P(S_i-S_{n-i}^*P)h,h^{\prime} \rangle \\& =\langle
(S_i-PS_{n-i}^*)Ph,h^{\prime} \rangle \\&=\langle
D_{P^*}{B_i}^*D_{P^*}Ph,h^{\prime} \rangle \\&=\langle
{B_i}^*PD_{P}h,D_{P^*}h^{\prime} \rangle.
\end{align*}
Since $A_i$ is defined on $\mathcal D_P$, the desired identity is
obtained.\\

\noindent \textbf{(2).}
$D_{P}(A_iD_{P}+A_{n-i}^*D_{P}P)=(S_i-S_{n-i}^*P)+(S_{n-i}^*-P^*S_i)P=D_{P}^2S_i$,
by (\ref{funda-repeat}). Therefore,
$D_{P}S_i=A_iD_{P}+A_{n-i}^*D_{P}P$ because both LHS and RHS are
defined from $\mathcal H$ to $\mathcal D_{P}$.\\

\noindent \textbf{(3).} For $h \in \mathcal{H}$, we have
\begin{align*}
(S_iD_P-D_{P^*}B_{n-i}P)D_Ph
&= S_i(I-P^*P)h-(D_{P^*}B_{n-i}D_{P^*})Ph
\\
&= S_ih-S_iP^*Ph-(S_{n-i}^*-S_iP^*)Ph
\\
&= S_ih-S_iP^*Ph-S_{n-i}^*Ph+S_iP^*Ph
\\
&= (S_i-{S_{n-i}}^*P)h=(D_P A_i)D_Ph.
\end{align*}
The proof for the other identity is similar.\\

\noindent \textbf{(4).}
$(D_{P^*}{B_i}^*+PD_{P^*}{B_{n-i}})D_{P^*}=(S_i-PS_{n-i}^*)+P(S_{n-i}^*-S_iP^*)=S_iD_{P^*}^2$,
by (\ref{funda-repeat1}). Therefore,
$S_iD_{P^*}=D_{P^*}{B_i}^*+PD_{P^*}{B_{n-i}}$.\\

\noindent \textbf{(5).} We have that
\begin{align*}
S_i^*S_j & = S_i^*(S_{n-j}^*P+D_PA_jD_P) \quad [\text{by }
(\ref{funda-repeat})] \\
& = S_{n-j}^*S_i^*P+S_i^*D_PA_jD_P \\
&= S_{n-j}^*(S_{n-i}-D_PA_{n-i}D_P)+(D_PA_i^*+P^*D_PA_{n-i})A_jD_P
\\ & \quad \quad \quad \quad \quad \quad \quad [\text{by }
(\ref{funda-repeat}) \text{ and part-}(2) ] \\
& = S_{n-j}^*S_{n-i}-(D_PA_{n-j}^*+P^*D_PA_j)A_{n-i}D_P \\
& \quad +D_PA_i^*A_jD_P+P^*D_PA_{n-i}A_jD_P \quad [\text{by
part-}(2)] \\
& = S_{n-j}^*S_{n-i} +D_{P}(A_i^*A_j-A_{n-j}^*A_{n-i})D_{P}, \quad
[\text{ since } [A_j,A_{n-i}]=0 ].
\end{align*}
Therefore, $S_i^*S_j - S_{n-j}^*S_{n-i} =
D_{P}(A_i^*A_j-A_{n-j}^*A_{n-i})D_{P} $.\\

\textbf{(5).} We apply similar techniques as in the previous part.
Now
\begin{align*}
S_iS_j^* & = S_i(S_{n-j}P^*+D_{P^*}B_{j}D_{P^*}) \quad [\text{by }
(\ref{funda-repeat1})] \\
& = S_{n-j}S_iP^*+S_iD_{P^*}B_{j}D_{P^*} \\
&=
S_{n-j}(S_{n-i}^*-D_{P^*}B_{n-i}D_{P^*})+(D_{P^*}B_{i}^*+PD_{P^*}B_{n-i})B_{j}D_{P^*}
\\ & \quad \quad \quad \quad \quad \quad \quad [\text{by }
(\ref{funda-repeat1}) \text{ and part-}(3) ] \\
& = S_{n-j}S_{n-i}^*-(D_{P^*}B_{n-j}^*+PD_{P^*}B_{j})B_{n-i}D_{P^*} \\
& \quad +(D_{P^*}B_{i}^*+PD_{P^*}B_{n-i})B_{j}D_{P^*} \quad
[\text{by
part-}(3)] \\
& = S_{n-j}S_{n-i}^*
+D_{P^*}(B_{i}^*B_{j}-B_{n-j}^*B_{n-i})D_{P^*}, \quad [\text{
since } [B_{j},B_{n-i}]=0 ].
\end{align*}
Therefore, $S_iS_j^*-S_{n-j}S_{n-i}^* =
D_{P^*}(B_{i}^*B_{j}-B_{n-j}^*B_{n-i})D_{P^*}$.

\qed

\vspace{0.2cm}

\noindent \textbf{\textit{Proof of Equations} (\ref{eqn:New01}).}\\

\noindent $(a_1).$ We apply part-(2) of Lemma \ref{funda-properties} to obtain
\begin{align*}
A_iA_{n-j}^*D_{P}+A_{n-i}^*D_{P}S_j & =A_iA_{n-j}^*D_{P}+
A_{n-i}^*(A_jD_P+A_{n-j}^*D_PP) \\& =(A_iA_{n-j}^*+A_{n-i}^*A_j)D_{P}+A_{n-i}^*A_{n-j}^*D_{P}P.
\end{align*}
Similarly
$A_jA_{n-i}^*D_{P}+A_{n-j}^*D_{P}S_i=(A_jA_{n-i}^*+A_{n-j}^*A_i)D_{P}+A_{n-j}^*A_{n-i}^*D_{P}P$.
So, the desired equality follows from the relations (\ref{rel:01}).\\

\noindent $(a_2).$ By part-(3) of Lemma \ref{funda-properties} we obtain
$
S_iD_{P^*}{B_{n-j}} =(D_{P^*}{B_i}^*+PD_{P^*}{B_{n-i}}){B_{n-j}}.
$
Similarly we have
$
S_jD_{P^*}{B_{n-i}} =(D_{P^*}{B_j}^*+PD_{P^*}{B_{n-j}}){B_{n-i}}.
$
So we have
\begin{align*}
S_iD_{P^*}{B_{n-j}} - S_jD_{P^*}{B_{n-i}} & = D_{P^*}
({B_i}^*{B_{n-j}}- {B_j}^*{B_{n-i}})
\\ & \quad \quad -PD_{P^*}({B_{n-j}}{B_{n-i}}-{B_{n-i}}{B_{n-j}})
\\&
= D_{P^*} ({B_i}^*{B_{n-j}}- {B_j}^*{B_{n-i}}) \quad
[\text{ by part-(1) of } (\ref{rel:01}) ] \\
& =D_{P^*} ({B_{n-j}}{B_i}^*- {B_{n-i}}{B_j}^*) \quad [\text{ by
part-(2) of } (\ref{rel:01})].
\end{align*}
The desired identity follows from here.\\

\noindent $(a_3).$  Let us denote
\begin{align*}
X_1 & = -A_iA_{n-j}^*P^*+A_{n-i}^*D_PD_{P^*}B_{n-j}-A_{n-i}^*P^*B_{j}^* \,, \\
X_2 & =
-A_jA_{n-i}^*P^*+A_{n-j}^*D_PD_{P^*}B_{n-i}-A_{n-j}^*P^*B_{i}^*\,.
 \end{align*}
 Then,
\begin{align*}
&D_{P}(X_2-X_1)D_{P^*} \\& =
D_{P}(A_iA_{n-j}^*-A_jA_{n-i}^*)P^*D_{P^*} +
D_{P}({A_{n-i}^*}P^*{B_j}^* - A_{n-j}^*P^*{B_i}^*)D_{P^*}
\\& \quad +
D_{P}(A_{n-j}^*D_PD_{P^*}B_{n-i}-A_{n-i}^*D_PD_{P^*}B_{n-j})D_{P^*}
\\& = D_{P}(A_iA_{n-j}^*-A_jA_{n-i}^*)P^*D_{P^*} + D_{P}({A_{n-i}^*}P^*{B_j}^* - A_{n-j}^*P^*{B_i}^*)D_{P^*}\\&
\quad  + (S_{n-j}^*-P^*S_j)(S_{n-i}^*-S_iP^*) \\& \quad
-(S_{n-i}^* -P^*S_i)(S_{n-j}^*-S_jP^*) ,\; [\text{ by }
(\ref{funda-repeat}), (\ref{funda-repeat1})] \\& =
D_{P}(A_iA_{n-j}^*-A_jA_{n-i}^*)D_{P}P^* + D_P(P^*B_{n-i}B_{j}^* -
P^*B_{n-j}B_{i}^*)D_{P^*} \\& \quad +
(S_{n-i}^*S_j-S_{n-j}^*S_i)P^*+P^*(S_iS_{n-j}^*-S_jS_{n-i}^*),\;
[\text{Lemma } \ref{funda-properties}, \text{part-(1)}]
\\& = D_{P}(A_iA_{n-j}^*-A_jA_{n-i}^*)D_{P}P^* + P^*D_{P^*}(B_{n-i}B_{j}^* - B_{{n-j}}B_{i}^*)D_{P^*} \\&
\quad +\{
(S_{n-i}^*S_j-S_{n-j}^*S_i)P^*+P^*(S_iS_{n-j}^*-S_jS_{n-i}^*)
\}\\& =D_{P}(A_iA_{n-j}^*-A_jA_{n-i}^*)D_{P}P^* +
P^*D_{P^*}(B_{n-i}B_{j}^* - B_{n-j}B_{i}^*)D_{P^*} \\& \quad + \{
D_P(A_{n-i}^*A_j-A_{n-j}^*A_i)D_PP^*+P^*D_{P^*}(B_{n-j}B_{i}^*-B_{n-i}B_{j}^*)D_{P^*}
\}.
\end{align*}
The last equality follows from part-(4) and part-(5) of Lemma
\ref{funda-properties}. Again from part-(2) of (\ref{rel:01}), we
have that
\begin{itemize}
\item[(i)]
$A_{n-i}^*A_j-A_{n-j}^*A_i=-(A_iA_{n-j}^*-A_jA_{n-i}^*)$
\item[(ii)] $B_{n-j}B_{i}^*-B_{n-i}B_{j}^*= - (B_{n-i}B_{j}^* -
B_{n-j}B_{i}^*). $
\end{itemize}
Therefore, $D_{P}(X_2-X_1)D_{P^*} =0$ and thus $X_1=X_2$.

\qed

\noindent \textbf{\textit{Proof of Equation} (\ref{eqn:New02}).}\\

For $d_0 \in \mathcal D_{P^*}$ we have the following:
\begingroup
\begin{align*}
& \quad \tilde{D}_i^*(d_0,0,0,\dots)
\\
&=
(D^*\tilde{C}_i^*C+C^*V_i^{'*}C)(d_0,d_1,\dots)
\\
&=
D^*\tilde{C}_i^* \left( D_{P^*}d_0 \oplus (-P^*d_0,0,\dots) \right)+C^*V_i^{'*} \left( D_{P^*}d_0 \oplus (-P^*d_0,0,\dots) \right)
\\
&=
D^*((B_{n-i}^*(I-PP^*)+P A_{n-i}P^*)d_0,0,0,\dots) \\
& \quad \quad + \; C^* \left( (S_{n-i}^*D_{P^*}-D_P A_{n-i}P^*)d_0 \oplus (-A_{i}^*P^*d_0,,\dots) \right)
\\
&=
(0,(B_{n-i}^*(I-PP^*)+P A_{n-i}P^*)d_0,,\dots) \\
&  \quad \quad  + \; ((D_{P^*}S_{n-i}^*D_{P^*}-D_{P^*}D_P A_{n-i}P^*+P A_i^*P^*)d_0,0,\dots)
\\
&=
( (D_{P^*}(S_{n-i}^*D_{P^*}-D_P A_{n-i}P^*)+P A_i^*P^*)d_0,
\\& \quad \quad \quad \quad  (B_{n-i}^*(I-PP^*)+P A_{n-i}P^*)d_0,0,\dots )
\\
&=
\left( ((I-PP^*)B_i+P A_i^*P^*)d_0,(B_{n-i}^*(I-PP^*)+P A_{n-i}P^*)d_0,0,\dots \right) 
\\& \qquad \qquad  \qquad \qquad  \qquad \qquad \qquad \qquad \qquad  \quad [\text{ by part-(3) of Lemma \ref{funda-properties}}]
\\
&=
(B_id_0,B_{n-i}^*d_0,0,\dots) \qquad \qquad \qquad \qquad \quad [\text{ by part-(1) of Lemma \ref{funda-properties}}]
\end{align*}
\endgroup

\qed

\noindent \textbf{\textit{Proof of Equation (\ref{eqn:New03}).}}\\
For $(c_0,c_1,c_2,\dots)$ and $(d_0,d_1,d_2,\dots)$ in $l^2(\mathcal{D}_{P^*})$, we have
\begin{eqnarray*}
&&\langle(c_0,c_1,c_2,\dots),\tilde{D}_i^*(d_0,d_1,d_2,\dots) \rangle
\\
&=& \langle (c_0,c_1,c_2,\dots), (B_id_0,B_{n-i}^*d_0+B_id_1,B_{n-i}^*d_1+B_{i}d_2,\dots) \rangle
\\
&=&
\langle c_0,B_i d_0 \rangle + \langle c_1,B_{n-i}^*d_0+B_i d_1 \rangle + \langle c_2,B_{n-i}^*d_1+B_i d_2 \rangle + \cdots
\\
&=&
\langle B_i^*c_0+B_{n-i}c_1,d_0 \rangle + \langle B_i^*c_1+B_{n-i}c_2,d_1 \rangle + \langle B_i^*c_2+B_{n-i}c_3, d_2 \rangle + \cdots
\\
&=&
\langle (B_i^*c_0+B_{n-i}c_1,B_i^*c_1+B_{n-i}c_2,B_i^*c_2+B_{n-i}c_3,\dots), (d_0,d_1,d_2,\dots) \rangle.
\end{eqnarray*}

\qed

\vspace{0.4cm}

\noindent \textbf{Acknowledgement.} The author is thankful to Nitin Tomar for pointing out an inaccuracy in the proof of a lemma in the previous version of this article.

\end{document}